\documentclass[12pt]{amsart}
\usepackage{amsmath}
\usepackage{amssymb}
\usepackage{amstext}
\usepackage{xcolor}
\usepackage{hyperref}
\usepackage{mathtools}
\usepackage{dsfont}

\mathtoolsset{showonlyrefs}

\usepackage{tikz}

\newtheorem{theorem}{Theorem}[section]

\newtheorem{lemma}[theorem]{Lemma}
\newtheorem{proposition}[theorem]{Proposition}
\newtheorem{definition}[theorem]{Definition}
\newtheorem{remark}[theorem]{Remark}

\DeclareMathOperator{\Ent}{Ent}
\DeclareMathOperator{\Id}{Id}
\DeclareMathOperator{\Hess}{Hess}

\DeclareMathOperator{\id}{id}

\DeclareMathOperator{\Image}{Im}
\DeclareMathOperator{\LSI}{LSI}

\newcommand{\R}{\mathbb{R}}
\newcommand{\T}{\mathbb{T}}

\begin{document}

 \title[Towards the quantitative hydrodynamic limit]{Toward a quantitative theory of the hydrodynamic limit}

\author{Deniz Dizdar}
\address{Universit\'e de Montr\'eal}
\email{deniz.dizdar@umontreal.ca}

\author{Georg Menz}
\address{University of California, Los Angeles}
\email{gmenz@math.ucla.edu}

\author{Felix Otto}
\address{Max Planck Institute for Mathematics in the Sciences, Leipzig, Germany}
\email{Felix.Otto@mis.mpg.de.}

\author{Tianqi Wu}
\address{University of California, Los Angeles}
\email{timwu@ucla.edu}

\date{\today}

 \maketitle

 \begin{abstract}
   This article provides non-trivial technical ingredients for the article \emph{The quantitative hydrodynamic limit of the Kawasaki dynamics} by the same authors. In that work a quantitative version of the hydrodynamic limit is deduced using a refinement of the two-scale approach. In this work, we deduce the strict convexity of the coarse-grained Hamiltonian, a uniform logarithmic Sobolev inequality for the canonical ensemble and the convergence of the coarse-grained Hamiltonian to the macroscopic free energy of the system. We deduce those results following an approach developed by Grunewald, Otto, Villani and Westdickenberg. Because in our setting the associated coarse-graining operator is non-local, the arguments are much more subtle and need additional ingredients like the Brascamp-Lieb inequality and a multivariate local central-limit theorem.
 \end{abstract}

\begin{footnotesize}
\end{footnotesize}

\tableofcontents

\newpage

\section*{Notations and conventions}

In this article we use the following conventions and notation.

\begin{itemize}
\item We use the letter~$C$ to denote a universal generic constant~$0< C< \infty$ that is independent of the dimension~$N$ of the underlying lattice. The actual value of~$C$ may change from line to line and sometimes even within a line.
\item We denote with~$a \lesssim b$ that
  \begin{align*}
    a \leq C b.
  \end{align*}
\item We denote with~$a \cdot b$ and~$|a|$ the standard Euclidean inner product and norm of the vectors~$a,b$.
\item We denote with~$\mathcal{P}(X)$ the set of probability measures on a measurable space~$X$.
\item We denote with~$\mathcal{L}^L$ the Hausdorff measure of dimension~$L$. If there is no source of confusion we use $dx$ is used as a shorthand notation for the Hausdorff measure of appropriate dimension.
\item $\|\cdot \|$ denotes the norm of a linear operator or multi-linear form with respect to given norms on the vector spaces.
\item $|\cdot|_{H^{1}}$ denotes the homogeneous~$H^1$ norm.
\item $[M]:= \left\{ 1 , \ldots, M \right\}$.
 \end{itemize}

\section{Introduction and main results}

A fundamental observation in probability theory is that patterns emerge out of randomness on large scales. For example, the law of large numbers states that for a sequence of i.i.d. random variables~$X_i$ the normalized sum converges to the mean i.e.~
\begin{align}\label{e_slln}
  \lim_{N \to \infty} \frac{1}{N} \sum_{i=1}^N X_i \overset{\mbox{\tiny a.s.}}{=} \mathbb{E}[X_1].
\end{align}
We are interested in dynamic manifestations of this fundamental principle, particularly in the hydrodynamic limit of the Kawasaki dynamics. The Kawasaki dynamics is a natural spin-exchange dynamic on a one-dimensional lattice-system. The hydrodynamic limit for the Kawasaki dynamics was first deduced by~\cite{Fri89} and~\cite{GPV}. It states that, under a certain scaling limit, a microscopic stochastic evolution~$X(t)$ converges to a macroscopic deterministic evolution~$\zeta(t)$. The microscopic evolution~$X(t)$ is called Kawasaki dynamics. The evolution~$\zeta(t)$ is the solution of a nonlinear heat equation.\\

Our goal in this line of research is to make the qualitative theory of the hydrodynamic limit of the Kawasaki dynamics quantitative, i.e.~to derive explicit error bounds. In the companion article~\cite{DMOW18a} of this work we discuss the broader picture  of the hydrodynamic limit and deduce a quantitative version of the hydrodynamic limit. In this article, we provide some non-trivial ingredients for the argument in~\cite{DMOW18a}.\\

A first step toward a quantitative theory of the hydrodynamic limit was taken in~\cite{GORV}. There, a new method was introduced to deduce the hydrodynamic limit, called the two-scale approach. The two scale-approach could in principle yield quantitative error estimates for the hydrodynamic limit. However, only the main estimates were quantitative in~\cite{GORV} and the hydrodynamic limit was still deduced on a qualitative level. With some additional work one could make the result of~\cite{GORV} on the hydrodynamic limit quantitative. However, the estimates would not be optimal. The quantitative hydrodynamic limit was deduced in the companion article~\cite{DMOW18a}, with error rates that are better than the rates one would obtain from~\cite{GORV}.\\

Both articles~\cite{DMOW18a} and~\cite{GORV} use the two-scale approach. The main idea of the two-scale approach is to artificially introduce a mesoscopic scale between the microscopic and the macroscopic scale. The mesoscopic scale is defined via a coarse-graining operator~$P$ that projects the microscopic dynamic~$X(t)$ onto mesoscopic observables~$PX(t)$. Then, a natural dynamic~$Y(t)$ is defined on the mesoscopic scale. The hydrodynamic limit is deduced in two steps: In the first step one shows that~$X(t)$ converges to~$Y(t)$ and in the second step one shows that~$Y(t)$ converges to~$\zeta (t)$. The main difference between the approach of~\cite{GORV} and~\cite{DMOW18a} is the definition of the coarse-graining operator~$P$. In~\cite{GORV}, the coarse-graining operator~$P$ projects onto piece-wise constant functions, whereas in~\cite{DMOW18a} the operator~$P$ projects onto splines. Working with splines is more challenging than working with piece-wise constant functions. The reason is that under the projection blocks become dependent. However, the estimates also improve because splines are smoother than piece-wise constant functions.\\

In this article we provide some nontrivial ingredients that are used in~\cite{DMOW18a}. The main results of this article are:
  \begin{itemize}
    \item Strict convexity of the coarse-grained Hamiltonian (see Theorem~\ref{p_strict_convexity_coarse_grained_Hamiltonian}).
  \item Uniform logarithmic Sobolev inequality for the conditional Gibbs measure given a mesoscopic profile (see Theorem~\ref{p_LSI_conditional}).
\item Convergence of the gradient of the coarse-grained Hamiltonian to the gradient of the  macroscopic free energy (see Theorem~\ref{p_convergence_meso_to_macro_free_energy}).
  \end{itemize}

This article emerged from Deniz Dizdar's diploma thesis~\cite{Dizdar} and some of the contents overlap with~\cite{Dizdar}. Compared to~\cite{Dizdar}, the proof of Theorem~\ref{p_convergence_meso_to_macro_free_energy} is completely new.\\

The rest of the introduction is organized in the following way. In Section~\ref{s_motivating_result}, we explain the main result of~\cite{DMOW18a}, namely the quantitative hydrodynamic limit of the Kawasaki dynamics. There, we also introduce notation and the set-up for the rest of this article. In Section~\ref{s_main_results} we discuss the main results of this article.

\subsection{Motivating result: The quantitative hydrodynamic limit of the Kawasaki dynamics}\label{s_motivating_result}

In this section we outline the main result of~\cite{DMOW18a}. For the reader's convenience we skip some details.\medskip

Let us start with defining the Kawasaki dynamics on the microscopic lattice~$\left\{1, \ldots, N \right\}$. For this purpose, let us introduce the Hamiltonian~$H : \mathbb{R}^N \to \mathbb{R}$ of the system, which is given by
\begin{equation}
H_N(x) \,=\, \sum_{n=1}^N \psi(x_n). \label{hamiltonian}
\end{equation}
We assume that the function~$\psi: \mathbb{R} \to \mathbb{R}$ is of the form
\begin{equation}\label{e_structure_of_ss_potential}
 \psi(x_n)\,=\,  \frac{1}{2}\,x_n^2 + ax_n+ \delta\psi(a_n),
\end{equation}
where the function $\delta\psi$ is bounded in $C^2(\mathbb{R})$:
\begin{equation}\label{e_assumptions_on_ss_perturbation}
\|\delta\psi\|_{L^{\infty}(\mathbb{R})} < C \:\: \text{and} \:\:\|\frac{d^2}{dx^2}\,
\delta\psi\|_{L^{\infty}(\mathbb{R})} < C.
\end{equation}
We also define the associated log moment generating function
\begin{equation}\label{psi^*}
\psi^*(\sigma) = \ln \int_{\mathbb{R}} \exp \left( \sigma z - \psi (z) \right) dz.
\end{equation}
By basic properties of the function $\psi^*$ (cf. \eqref{e_d_mu_m} from below
), we choose $a$ such that: 
\begin{align}\label{e_ss_mean_0}
	\frac{\int_{\mathbb{R}}  z \exp (- \psi(z)) dz}{\int_{\mathbb{R}} \exp (- \psi(z)) dz } =0.\\
\end{align}
The function $\psi$ may be non-convex and it helps think about the function $\psi$ as a double well-potential (see Figure~\ref{f_double_well}). \\

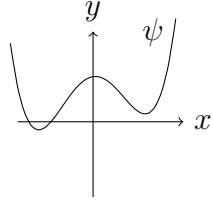
\begin{figure}[t]
  \centering
  \begin{tikzpicture}[scale=1]
      \draw[->] (-1,0) -- (1.2,0) node[right] {$x$};
      \draw[->] (0,-1) -- (0,1.2) node[above] {$y$};
      \draw[scale=0.5,domain=-2.2:2.2,smooth,variable=\x,black] plot
      ({\x},{0.3*(\x*\x-2)*(\x*\x-2)+0.15*\x});
      \draw (0.8,1.2) node {$\psi$};
    \end{tikzpicture}
  \caption{Double-well potential~$\psi$}\label{f_double_well}
\end{figure}

The Kawasaki dynamics~$X_t$ is given by the solution of the SDE
\begin{align}\label{e_Kawasaki_dynamics}
  dX_t = - A \nabla H (X_t) dt + \sqrt{2A} dB_t.
\end{align}
Here,~$B_t$ denotes a standard Brownian motion on $\mathbb{R}^N$, ~$A$ denotes the second order difference operator of the periodic rescaled lattice~$\left\{ \frac{1}{N}, \ldots, 1  \right\}$. More precisely, the operator~$A$ is given by the $N\times N-$matrix
\begin{equation}\label{e_def_A}
A_{i,j} := N^2(-\delta_{i,j-1} + 2\delta_{i,j} - \delta_{i,j+1}),
\end{equation}
where we use the convention that~$N+1=1$ and~$0=N$. It follows from the structure of the operator~$A$ that the Kawasaki dynamics~~\eqref{e_Kawasaki_dynamics} conserves the mean spin of the system. For the purpose of stating the results of the companion article \cite{DMOW18a}, we restrict the state space~$X_N = \mathbb{R}^N$ of the Kawasaki dynamics~$X_t$ to the hyperplane of mean $0$
\begin{equation}
\label{e_definition_of_X_N_M}
X_{N, 0} := \left\{  x \in \mathbb{R}^N, \ \frac{1}{N} \sum_{i=1}^N x_i =0   \right\}.
\end{equation}
\begin{remark} We use a different notation than in the companion article \cite{DMOW18a}, because we want to mainly work with the full state space $\mathbb{R}^N$ in this article. 
\end{remark}

The hydrodynamic limit of the Kawasaki dynamics~$X_t \in X_{N, 0}$ means that as~$N \to \infty$ the stochastic dynamics~$X_t$ on the periodic rescaled lattice~$\left\{\frac{1}{N}, \frac{2}{N}, \ldots, 1 \right\}$ converges to a deterministic dynamic~$\zeta (t) $ defined on the one-dimensional torus~$\mathbb{T}=[0,1)$. To do this, we embed the spaces $X_{N, 0} \subset X_N$ into the space~$L^{2}(\mathbb{T})$ by identifying the vector~$x \in X_N$ with its corresponding step function on the interval~$[0,1]$. More precisely, given~$x \in X_N$, we identify it with the step function
	\begin{equation} \label{def_step_functions}
	x(\theta) = x_j, \hspace{1cm} \theta \in \left[ \frac{j-1}{N}; \frac{j}{N} \right).
	\end{equation}
	Then the space $X_N$ (resp. $X_{N, 0}$) is identified with the space of piecewise constant functions on $\mathbb{T} = \R/\mathbb{Z}$ (resp. with zero mean), i.e.
	\begin{align}\label{e_embedding_X}
	X_{N} = &\left\{x : \mathbb{T} \rightarrow \R; \hspace{2mm} x \text{ is constant on } \left[\frac{j-1}{N}; \frac{j}{N} \right), \hspace{2mm} j = 1,..,N\right\}. \\
	X_{N, 0} = &\left\{x \in X_N: \quad \int_0^1 x(\theta) d\theta = 0 \right\}.
	\end{align}

It turns out the $L^2$ norm is not well-adapted to our problem, since it is too sensitive to local fluctuations. Instead, we work with the $H^{-1}$-norm which is defined in the following way: If $f : \mathbb{T} \rightarrow \R$ is a locally integrable function with zero mean, then

\begin{equation}
||f||_{H^{-1}}^2 := \int_{\T}{w(\theta)^2d\theta}, \hspace{5mm} w' = f, \hspace{3mm} \int_{\T}{w(\theta)d\theta} = 0.
\end{equation}\medskip

Now, let us us formulate the main result of~\cite{DMOW18a}.

\begin{theorem}[Quantitative hydrodynamic limit of the Kawasaki dynamics] \label{thm_lim_hydro}
We assume that the single-site potential $\psi$ satisfies~\eqref{e_structure_of_ss_potential} and~\eqref{e_assumptions_on_ss_perturbation}. Let~$\mu$ denote the Gibbs measure on~$X_N$ associated to the Hamiltonian~$H$. More precisely, the probability measure~$\mu$ is absolutely continuous wrt.~the~$N-1$-dimensional Hausdorff measure~$\mathcal{L}^{N-1}$ and the Radon-Nikodym derivative of~$\mu$ is given by
\begin{align}\label{e_def_Gibbs_measure}
  \frac{d\mu}{d \mathcal{L}^{N-1}}(x) =  \frac{1}{Z} \exp \left( - H(x) \right) \qquad  x \in X_N,
\end{align} 
where~$Z$ is a normalizing factor making~$\mu$ a probability measure.
We assume that the initial distribution~$\nu$ of the Kawasaki dynamics~$X(0)$ is absolutely continuous wrt.~the Gibbs measure $\mu$ and satisfies the estimate
\begin{equation}\label{e_closness_initial_data}
 \Ent(\nu | \mu) := \int  \ln \left( \frac{d \nu}{d\mu} \right) d \nu \leq C_{\Ent} N 
\end{equation}
for some constant $C_{\Ent} > 0$. \newline
Let $\zeta$ denote the unique weak solution (for a precise definition see~\cite{DMOW18a}) of the equation
\begin{equation} \label{hydro_equn}
\frac{\partial \zeta}{\partial t} = \frac{\partial^2}{\partial \theta^2}\varphi'(\zeta)
\end{equation}
with initial condition $\zeta(0,\cdot) = \zeta_0$ and $\varphi$ is the Legendre-Fenchel transform of the function $\psi^*$ (cf. \eqref{psi^*})
\begin{equation} \label{average_hydro_potential}
\varphi(m) = \underset{\sigma \in \R}{\sup} \hspace{1mm} \left\{ \sigma m - \psi^*(\sigma) \right\}.
\end{equation}
Then there is a constant~$0<C< \infty$ depending only on the constants appearing in~\eqref{e_assumptions_on_ss_perturbation} such that for any $T > 0$ 
\begin{align}
  \label{e_quantitative_estimate_1}
   \sup_{0 \leq t \leq T} & \mathbb{E} \left[  |X(t) - \zeta(t)|_{H^{-1}}^2 \right]    \\
& \lesssim  \   \mathbb{E} \left[ |x(0) - \zeta (0)|_{H^{-1}}^2 \right] + \frac{1}{N^{\frac{2}{3}}} \Big[   T +  C_{\Ent}+  |\zeta (0)|_{L^2}^2 +  1 \Big].
\end{align}
\end{theorem}

\subsection{Main results of this article}\label{s_main_results}

The main results of this article are important ingredients for deducing the quantitative hydrodynamic limit via the two-scale approach (see the companion article~\cite{DMOW18a} and Theorem~\ref{thm_lim_hydro}). A central object in the two-scale approach is the coarse-graining operator~$P$. In~\cite{GORV}, the coarse-graining operator~$P$ projects the state~$x \in X_N$ onto piecewise constant functions on the interval~$[0,1]$. More relevant to this article is a more subtle coarse-graining operator~$P$ used in~\cite{DMOW18a}:~$P$ projects $x \in X_N$ onto a spline function on the torus~$\mathbb{T}$.

\begin{definition}[Definition of the coarse-graining operator~$P$]\label{d_coarse_graining_operator_P}
  For $M \in \mathbb{N}$, let $Y_M$ be the space of spline functions of degree~$L \in \mathbb{N}$  on the torus $\mathbb{T}$ corresponding to the mesh $\left\{\frac{m}{M}\right\}_{m=1,...,M}$. More precisely, if~$L \geq 1$
\begin{align*}
  Y_M := &\left\{y\in C^{L-1}(\mathbb{T})|\,\forall m\in [M]:    y|_{\left(\frac{m-1}{M},\frac{m}{M}\right)}\, \text{polynomial of degree $\leq$ L}\right.\},
\end{align*}
and if $L\,=\,0$:
\begin{equation*}
  Y_M := \left\{y \in L^2(\mathbb{T})|\,\forall m \in [M]: \text{y constant on}
  \,\left(\frac{m-1}{M},\frac{m}{M}\right)\right\}.
\end{equation*}
We endow $Y_M$ with the inner product inherited from $L^2(\mathbb{T})$. We define the coarse graining operator $P: L^2(\mathbb{T})\rightarrow Y_M\subset L^{2}(\mathbb{T})$ as the $L^2$-orthogonal projection onto $Y_M \subset L^2(\mathbb{T})$. 
\end{definition}

\begin{remark}
	Only splines of second order,~i.e.~$L=2$, are needed in the companion article~\cite{DMOW18a}. However, we state the main results of this article for splines of general order~$L \in \mathbb{N}$. From now on, the polynomial degree $L$ will be fixed for the entire article and we will not always mention the dependencies of constants on $L$ explicitly. By choosing~$L=0$ one recovers the coarse-graining operator~$P$ of~\cite{GORV}. Using~$L=2$ allows in better error estimate in the quantitative hydrodynamic limit compared to~\cite{DMOW18a}. However, the analysis becomes a lot more complex.
\end{remark}

From now on, we assume $N = KM$ for $K\in \mathbb{N}$, and $K$ is large enough so that $P$ is surjective (a consequence of Lemma \ref{p_invertible_PNPt} below). The coarse-graining operator~$P$ gives rise to another central object in the two-scale approach, the coarse-grained Hamiltonian~$\bar H$. It defines a natural energy on the mesoscopic scale~$Y_M$.

\begin{definition}[Coarse-grained Hamiltonian~$\bar H$.]
  The coarse-grained Hamiltonian~$\bar H: Y_M \to \mathbb{R}$ is given by 
  \begin{align}\label{e_def_coarse_grained_Hmailtonian}
   \bar H(y) := \bar H_{Y_M}(y) :=  - \frac{1}{N} \ln \int_{\left\{x \in X_N  :  Px =y \right\}} \exp \left( - H(x) \right) dx,
  \end{align}
where~$dx$ denotes the Hausdorff measure on the space~$\left\{x \in X_N  :  Px =y \right\}$.
\end{definition}

Now, let us turn to the main results of this article. The first main result is that the coarse-grained Hamiltonian is uniformly strictly convex for $K$ large enough.
\begin{theorem}[Strict convexity of~$\bar H$]\label{p_strict_convexity_coarse_grained_Hamiltonian}
There are constants~$0< \lambda, \Lambda, K^* < \infty$ such that for all~$K \geq K^*$,~$M$ and~$y \in \mathbb{R}^M$ it holds
\begin{align*}
2 \lambda \Id_{Y_M} \leq  \Hess_{Y_M} \bar H(y) \leq 2 \Lambda \Id_{Y_M}
\end{align*}
in the sense of quadratic forms.
\end{theorem}
We state the proof of Theorem~\ref{p_strict_convexity_coarse_grained_Hamiltonian} in Section~\ref{s_convexification}. One should compare Theorem~\ref{p_strict_convexity_coarse_grained_Hamiltonian} to the similar statement of Lemma 29 in~\cite{GORV}. The situation here is more subtle. The reason is that in~\cite{GORV}, the mesoscopic observables are also piecewise constant functions and therefore local functions. In our setting the mesoscopic observables are given by continuous splines which are non-local functions. This introduces additional interactions between blocks. We work around this obstacle by first deducing the strict convexity for mesoscopic observables that are piecewise polynomials of degree~$L$. In the jargon of numerical analysis those functions are referred to as discontinuous Galerkin functions. From this we deduce the strict convexity of the coarse-grained Hamiltonian~$\bar H$ by using the fact that marginals of log-concave measures are again log-concave, which is a well-known consequence of the Brascamp-Lieb inequality.\medskip

Let us turn to the second main result, namely the uniform LSI for the conditional measure~$\mu(dx |P x=y)$. The LSI was first described by Gross~\cite{Gro75}. 
It yields strong concentration properties and also gives exponential convergence to equilibrium of the associated diffusion process. For a good introduction to the LSI we refer the reader to the books~\cite{R,BGL14} and the overview article~\cite{L}.\medskip

In the article\cite{DMOW18a}, the uniform LSI is used to establish the convergence of the microscopic Kawasaki dynamic to the mesoscopic dynamics. More precisely, a consequence of the LSI is that the Kawasaki dynamics equilibrates very fast on small spatial scales. This observation yields that~$P\nabla H (X(t))$ is close to~$P \mathbb{E} \left[ \nabla   H(X(t))| P X(t) =Y(t)\right] $, which itself is close to~$\nabla \bar H(Y(t))$.\\

Now, let us define the conditional measure~$\mu(dx |P x=y)$ which describes the fluctuations of the Gibbs measure~$\mu$ (cf.~\eqref{e_def_Gibbs_measure}) around a mesoscopic profile~$y \in Y_M$.
 
\begin{definition}[Disintegration of the canonical ensemble~$\mu$]\label{d_disintegration_canonical_ensemble}
  The coarse-graining operator~$P: L^2(\mathbb{T})\to Y_M$ introduces a decomposition of the canonical ensemble~$\mu$ into conditional measure~$\mu(dx| P x =y)$ and marginal measures~$\bar \mu(dy)$. More precisely the measures $\mu(dx| P x =y)$,~$y \in Y_M$, and~$\bar \mu(dy)$ are defined by the relation
  \begin{align*}
    \int f (x) \mu (dx) = \int \int f(x) \mu(dx| P x =y) \bar \mu(dy) 
  \end{align*}
for any test function~$f:X_N \to \mathbb{R}$. The conditional measure~$\mu(dx| P x =y)$ is a probability measure on the space
\begin{align*}
 P^{-1} \left\{y  \right\} = \left\{ x \in X_N\ | \ P x =y\right\} \subset X_N
\end{align*}
that is absolutely continuous wrt.~the Hausdorff measure~$dx$. Its Radon-Nikodym derivative is given by
\begin{align}
  \label{e_d_conditional_can_ensemble}
  \frac{d \mu}{dx} (x) = \frac{1}{Z} \ \mathds{1}_{\left\{ P x =y \right\}}(x) \ \exp \left(  - H(x)  \right).
\end{align}
For convenience, we also may write~$\mu(dx|y)$ instead of~$\mu(dx|P x =y)$.

The marginal measure~$\bar \mu$ is a probability measure on the space $Y_M$ that is absolutely continuous wrt.~the Hausdorff measure~$dy$. Its Radon-Nikodym derivative is given by
\begin{align}
  \label{e_d_marginal}
  \frac{d \bar \mu}{d y} (y) = \frac{1}{Z}  \exp\left( - N \bar H(y)  \right),
\end{align}
where~$\bar H$ is the coarse-grained Hamiltonian given by~\eqref{e_def_coarse_grained_Hmailtonian}.
\end{definition}

Now, let us formulate the second main result of this article.

\begin{theorem}[Uniform LSI for the conditional measure~$\mu(dx |P x  =y)$]\label{p_LSI_conditional}
 The conditional measure~$\mu(dx |P x =y)$ given by~\eqref{e_def_Gibbs_measure} satisfies a LSI with constant~$\varrho>0$ uniform in the system size~$N$ and the mesoscopic profile~$y$. More precisely, if~$f: Y_M \to \mathbb{R}$ is a nonnegative test function that satisfies~$\int f (x) \mu (dx | P x =y)=1$ then
 \begin{align}
\label{e_d_LSI}   \mbox{} \\
 \Ent \left( f \mu( dx | P x =y) | \mu( dx | P x =y) \right) \leq \frac{1}{\varrho} \int \frac{|\nabla_{||} f(x)|^2}{f(x)} \mu( dx | P x =y),
 \end{align}
where~$|\nabla_{||}f|$ is the norm of the gradient on~$\ker P$ wrt.~the standard Euclidean structure.
\end{theorem}
The statement of Theorem~\ref{p_LSI_conditional} should be compared to~\cite[Theorem 14]{GORV}, where a similar statement was deduced for the case~$L=0$. We proof Theorem~\ref{p_LSI_conditional} in Section~\ref{s_convexification} by extending the approach of~\cite{GORV}. The uniform LSI can not directly be deduced as in~\cite{GORV} from the two-scale criterion (see Lemma~\ref{OR} below). The reason is that the projection~$P$ onto the spline space~$Y_M$ is not local. We get around this obstacle by first deducing a uniform LSI for another conditional measure. More precisely, in Theorem~\ref{LSIDG} below we use the two-scale criterion to deduce the uniform LSI for the conditional measures~$\mu(dx|Q_Mx=y)$, where~$Q_M: L^2(\mathbb{T}) \to Y_M^{DG}$ denotes the orthogonal projection in~$L^2$ onto the subspace~$Y_M^{DG}$ of discontinuous Galerkin functions (see~\eqref{d_Y_DG} ff.~below). Then, we use two scale criterion, Theorem~\ref{LSIDG} and the strict convexity of the coarse-grained Hamiltonian~$\bar H$ to show that the conditional measures~$\mu(dx |P x =y)$ satisfy a uniform LSI.\\

The third and last main result of this article is the convergence of the coarse-grained Hamiltonian~$\bar H$ to the macroscopic free energy~$\mathcal{H}: L^2 (\mathbb{T}) \to \mathbb{R}$, which is defined by
\begin{align}
  \label{e_macroscopic_free_energy}
  \mathcal{H} (\zeta) = \int_{\mathbb{T}} \varphi (\zeta(\theta)) d \theta ,
\end{align}
where the function~$\varphi$ is given by~\eqref{average_hydro_potential}.

\begin{theorem}[Convergence of the gradient of the coarse-grained Hamiltonian to the gradient of the macroscopic free energy] \label{p_convergence_meso_to_macro_free_energy}
There is a constant~$K_0$ such that for all~$K \geq K_0$ and~for all~$\zeta \in L^2$

  \begin{align}
    \label{e_estimate_meso_to_macro_free_energy}
    \left| \nabla \bar H (P\zeta) -   \nabla \mathcal{H}(\zeta)  \right|_{L^2}  \lesssim   \left(\frac{1}{K} + \frac{1}{M}\right)(|\zeta|_{L^2}+|\zeta|_{H^1}) + \frac{1}{K} .
  \end{align}
\end{theorem}

\begin{remark}
 It follows from the definition the gradient that~$\nabla \mathcal{H} (\zeta) = \varphi' \circ \zeta$ for any~$\zeta \in L^{2}(\mathbb{T})$.

\end{remark}
We give the proof of Theorem~\ref{p_convergence_meso_to_macro_free_energy} in Section~\ref{s_convergence_free_energies}.

\section{Proof of Theorem~\ref{p_strict_convexity_coarse_grained_Hamiltonian}}\label{s_convexification}

This section is devoted to the proof of Theorem~\ref{p_strict_convexity_coarse_grained_Hamiltonian}. Since the spline functions in $Y_M$ are non-local due to continuity requirements (for $L \geq 1$), we will instead work with an intermediate space $Y_M^{DG} \supset Y_M$ of discontinuous Galerkin functions, prove the analogue of Theorem~\ref{p_strict_convexity_coarse_grained_Hamiltonian} for $Y_M^{DG}$ (see Theorem ~\ref{p_strict_convexity_bar_H_DG} below), and then transfer the result to $Y_M$. 

\begin{definition} \label{d_Y_pol_and_H}
The space
of discontinuous Galerkin functions of degree~$L \in \mathbb{N}$ is defined as
\begin{align} \label{d_Y_DG}
Y^{DG}_{M} &:= \left\{y\in L^2(\mathbb{T})|\,\forall m \in [M]:  y|_{\left(\frac{m-1}{M},\frac{m}{M}\right)}
                         \,\text{polynomial of degree $\leq$ L} \right\}.
\end{align}
We endow $Y_M^{DG}$ with the inner product inherited from $L^2(\mathbb{T})$. 
\end{definition}

\begin{definition}\label{d_H_DG}
  We denote with $Q_M: L^2(\mathbb{T}) \rightarrow Y_M^{DG}$
the orthogonal projection onto $Y_M^{DG}$ in $L^2(\mathbb{T})$. The coarse-grained Hamiltonian $\bar H_{Y_M^{DG}}: Y_M^{DG} \to \mathbb{R}$ associated to~$Q_M$ is given by
  \begin{align}\label{e_def_coarse_grained_Hamiltonian_DG}
   \bar H_{Y_M^{DG}} (y) := - \frac{1}{N} \ln \int_{\left\{x \in \mathbb{R}^N  :  Q_M x =y \right\}} \exp \left( - H_N(x) \right) dx.
  \end{align}
\end{definition}
For later use, we also define two notions of adjoints of $Q_M$: 
\begin{definition}[Two notions of adjoints to the coarse-graining operator $Q_M$]\label{d_adjoint_NQMt} Restrict the coarse-graining operator $Q_M$ to $Q_M: X_N \rightarrow Y_M^{DG}$. First, define map~$Q_M^{t}: Y_M^{DG}  \to X_N$ as the adjoint to $Q_M$, when we endow $X_N = \mathbb{R}^N$ with the Euclidean inner product. It follows that the map $NQ_M^t:Y_M^{DG}  \to X_N$ is the adjoint to $Q_M$, when we endow $X_N$ with the inner product inherited from $L^2(0, 1)$. From this it also follows the map $NQ_M^t$ is the $L^2$-orthogonal projection of $Y_M^{DG}$ onto $X_N$, explicitly given by
	\begin{align}\label{e_adjoint_NQ_M^t_explicit}
	\left( NQ_M^t y \right)_i = N \int_{\frac{i-1}{N}}^{\frac{i}{N}} y(\theta) d \theta \quad \mbox{for} \quad i \in \{1, 2, \cdots, N\}.
	\end{align} 
\end{definition}

\begin{remark}
  Let us emphasize that the index~$M$~in $Y_M^{DG}$ and~$Q_M$ is used to denote the relevant number of intervales~$\left[ \frac{i-1}{M}, \frac{i}{M}  \right]$ that are used to decompose~$[0,1]$.
\end{remark}

\begin{theorem}[Strict convexity of~$\bar H_{Y_M^{DG}}$]  \label{p_strict_convexity_bar_H_DG}
There are constants~$0< \lambda, \Lambda, K^* < \infty$ such that for all~$K \geq K^*$,~$M$ and all~$y \in Y_M^{DG}$ it holds
\begin{align*}
2 \lambda \Id_{Y_M^{DG}} \leq  \Hess_{Y_M^{DG}} \bar H_{Y_M^{DG}}(y) \leq 2 \Lambda \Id_{Y_M^{DG}}
\end{align*}
in the sense of quadratic forms.
\end{theorem}
We give the proof of Theorem~\ref{p_strict_convexity_bar_H_DG} in Section~\ref{s_strict_convexity_bar_H_DG}. The argument is quite technical, making up the bulk of the proof of Theorem~\ref{p_strict_convexity_coarse_grained_Hamiltonian}. We follow a similar approach as in~\cite{GORV}: a local Cram\'er theorem and a local central limit theorem (CLT). The main difference here is that the local Cram\'er theorem has to be extended to canonical ensembles with multiple constraints, which means that we will have to use a multivariate CLT instead of a univariate CLT.\\


To finish the proof of Theorem~\ref{p_strict_convexity_coarse_grained_Hamiltonian}, we need to transfer the strict convexity from~$\bar H_{Y_M^{DG}}$ to ~$\bar H$. Since $\bar H_{Y_M^{DG}}$ and $\bar H$ are coarse-grained versions of $H$ on the levels of $Y_M^{DG}$ and $Y_M$, respectively, and $Y_M \subset Y_M^{DG}$, $\bar H$ is itself a coarse-grained version of $\bar H_{Y_M^{DG}}$, i.e.
  \begin{align} 
  N \bar H (y) & = - \ln \int_{z \in Y_M^\perp}\exp(- N \bar H_{Y_M^{DG}} (y + z)) \mathcal{L}^{LM} (dz)  N^{\frac{LM}{2}}. \label{e_bar_h_Y_in_terms_of_bar_h_Y^DG}
  \end{align}
where $Y_M^\perp:= \{y\in Y_M^{DG}: P y = 0\}$ and $\dim Y_M^\perp = LM$. (The factor $N^{\frac{LM}{2}}$ is due to the difference between the standard Euclidean structure on $X_N = \mathbb{R}^N$ and the $L^2$ structure on $X_N \subset L^2(\mathbb{T})$.)

\begin{lemma} \label{log-concave}
Let $W\oplus Z$ be an orthogonal decomposition of a finite dimensional Euclidean space. Suppose $F: W\oplus Z \rightarrow \mathbb{R}$ is a $C^2$ function such that $\int_{W\oplus Z} \exp(-F) < \infty$. Let $\bar{F}(z) := - \ln \int_{W} \exp(-F(w,z)) \,dw$. Let~$c \geq 0$, then it holds that
\begin{align} \label{e_log-concave_inherited}
    \operatorname{Hess}_{W\oplus Z} F > c  \operatorname{id}_{W\oplus Z} \ \Rightarrow \ \operatorname{Hess}_Z \bar{F} > c  \operatorname{id}_Z. \\
   \label{e_log-concave_inherited_2}
    \operatorname{Hess}_{W\oplus Z} F < c  \operatorname{id}_{W\oplus Z} \ \Rightarrow \ \operatorname{Hess}_Z \bar{F} < c  \operatorname{id}_Z.
\end{align}
\end{lemma}
In~\cite{BL}, it was shown in a very neat way that statement \eqref{e_log-concave_inherited} in Lemma~\ref{log-concave} is simple consequence of the well-known Brascamp-Lieb inequality 
). Statement \eqref{e_log-concave_inherited_2} follows from a straightforward computation. 

\begin{proof}[Proof of Theorem~\ref{p_strict_convexity_coarse_grained_Hamiltonian}]
  We apply Lemma~\ref{log-concave} with $Z=Y_{M}$, $W=Y^{\perp}_{M}$, and $F= N\bar{H}_{Y_M^{DG}}$. The hypotheses of~\eqref{e_log-concave_inherited} and \eqref{e_log-concave_inherited_2} are satisfied due to Theorem~\ref{p_strict_convexity_bar_H_DG}. This yields the statement of Theorem~\ref{p_strict_convexity_coarse_grained_Hamiltonian}. \end{proof}

\subsection{Proof of Theorem~\ref{p_strict_convexity_bar_H_DG}: Strict convexity of~$\bar H_{Y_M^{DG}}$}\label{s_strict_convexity_bar_H_DG} Our first step towards proving Theorem~\ref{p_strict_convexity_bar_H_DG} is reducing it to the case $M = 1$. Since the blocks are independent for functions in $Y_M^{DG}$, the space $Y_M^{DG}$ factors as an orthogonal direct sum
\begin{align}
	Y_{M}^{DG} \,=\,\bigoplus_{m=1}^M\, Y_1^{DG}\label{e_decomposition}
\end{align}
via the map
\begin{equation}
	\alpha \mapsto \oplus_{m=1}^M \alpha^{(m)} \label{e_identification}
\end{equation}
where the functions $\alpha^{(m)} \in Y_1^{DG}, m = 1, \cdots, M,$ are obtained by restricting the function $\alpha$ to subintervals. More precisely, for $m = 1, \cdots, M$, we define
\begin{equation}\label{e_restriction_to_subinterval}
	\alpha^{(m)}(\theta) := \alpha\left(\frac{m-1}{M} + \frac{\theta}{M}\right) \mbox{ for } \theta \in [0, 1].
\end{equation}
With this decomposition, the $L^2$ inner product on $Y_M^{DG}$ is given by
\begin{align}
	\langle \alpha,\beta\rangle_{Y_{M}^{DG}} 
	\,=\, \frac{1}{M}\,\sum_{m=1}^M \langle \alpha^{(m)}, \beta^{(m)} \rangle_{Y_1^{DG}}. \label{e_innerproduct} 
\end{align}
Given $x \in \mathbb{R}^N$, for $1\leq m \leq M$, denote $x^{(m)} := (x_{(m-1)K+1}\,,\,...\,,\,x_{mK}) \in \mathbb{R}^K$. Then it follows that in the sense of \eqref{e_identification}
\begin{align}
	Q_{M}x &= \oplus_{m=1}^M  Q_{1} x^{(m)}. \label{P3}
\end{align}
Equation \eqref{P3} implies that any fiber of the projection $Q_M$
\begin{align}\label{e_d_R_M_alpha}
	\mathbb{R}^N_{M,\alpha} := \{x\in \mathbb{R}^N|\,Q_{M}x =\alpha\}  
\end{align}
factors as an orthogonal direct sum of fibers of the projection $Q_1$
\begin{equation}
	\mathbb{R}^N_{M,\alpha}\,=\,\bigoplus_{m=1}^M\, \mathbb{R}^K_{1,\alpha^{(m)}}. \label{fibreprod}
\end{equation}
Consequently, the Hausdorff measure on $\mathbb{R}^N_{M,\alpha}$ is a product measure:
\begin{equation}
	\mathcal{H} (dx)\,=\, \bigotimes_{m=1}^M \,\mathcal{H} (dx^{(m)}). \label{Haussdorffac}
\end{equation}
Since the Hamiltonian $H_N$ of $\mu$ defined in \eqref{hamiltonian}
is also uncoupled, the observations above imply that the coarse-grained Hamiltonian $\bar H_{Y_M^{DG}}$ decomposes in the following way:

\begin{lemma}\label{p_decomp_bar_H_DG}
	For~$\beta \in Y_1^{DG}$ we define
	\begin{equation}
		\psi_K(\beta) = -\frac{1}{K} \,\ln\,\int_{\mathbb{R}^K_{1,\beta}} \exp (-H_K(x))\, dx.\label{psiJ}
	\end{equation}
	Then it holds that for~$\alpha \in Y_M^{DG}$ 
	\begin{align}\label{e_bar_H_pol_noninteracting}
		\bar H_{Y_M^{DG}} (\alpha) = \frac{1}{M} \sum_{m=1}^M \psi_K(\alpha^{(m)}).
	\end{align}
\end{lemma}
In light of ~\eqref{e_bar_H_pol_noninteracting}, proving Theorem~\ref{p_strict_convexity_bar_H_DG} reduces to proving the following:

\begin{theorem}\label{1block}
	There are constants~$0< \lambda, \Lambda, J^* < \infty$ such that for all~$J \geq J^*$, and all~$\beta \in Y_1^{DG}$ it holds
	\begin{align}\label{e_strict_convexity_psi_J}
		2 \lambda  \leq \Hess_{Y_{1}^{DG}}\psi_J(\beta) \leq 2 \Lambda ,
	\end{align}
	in the sense of quadratic forms.
\end{theorem}

\begin{proof}[Proof of Theorem~\ref{p_strict_convexity_bar_H_DG}] By \eqref{e_decomposition} and \eqref{e_innerproduct},
	\begin{align}
		\Hess_{Y_{M}^{DG}} \bar H_{Y_{M}^{DG}} (y) = M \Hess_{\oplus_{k=1}^M Y_1^{DG}} \bar H_{Y_{M}^{DG}} (y)
	\end{align}
	for all $y \in Y_{M}^{DG}$. Hence, a combination of~\eqref{e_bar_H_pol_noninteracting} and~\eqref{e_strict_convexity_psi_J} (with $J = K$) yields for large enough~$K$ the desired estimate 
	\begin{align}
		2 \lambda \id_{Y_{M}^{DG}} \leq \Hess_{Y_{M}^{DG}} \bar H_{Y_M^{DG}}(y) \leq 2 \Lambda  \id_{Y_{M}^{DG}}.
	\end{align}
\end{proof}

The rest of this section is devoted to the proof of Theorem~\ref{1block}. We prove uniform strict convexity of $\psi_J$ for large $J$ by showing that $\psi_J$ converges as $J\rightarrow \infty$ in the uniform $C^2$-topology to a uniformly strictly convex function. Namely,
this will be the Legendre-Fenchel transform of the function which, to each $\beta \in Y_1^{DG}$, associates the specific
free energy of a \emph{modified} grand canonical ensemble which makes the conditioning $Q_{1}x=\beta$
a \emph{typical} event. \medskip


Before we enter into the details, we give a sketch of the argument, which closely follows the argument in \cite{GORV}.  
Using Cram\'er's trick of exponential shift of measure,
we construct for each $\beta$ a product measure $\nu_{J,\beta}$ on $\mathbb{R}^J$ such that
\begin{itemize}
\item the law of each spin is an
``exponential shift"´ of the single-site measure (a perturbed standard Gaussian),
\item the expectation of $Q_{1}x$ under $\nu_{J,\beta}$ is equal to $\beta$, i.e.
the conditioning $Q_{1}x=\beta$ is a ``typical"´ event.
\end{itemize}
We refer to the product measure $\nu_{J,\beta}$ as the modified grand canonical ensemble for $\beta$.
The required shifts of spins can be parameterized by a variable $\hat{\beta}$ that is dual to $\beta$. \medskip

Because the single-site potential $\psi$ is quadratic plus a perturbation
that is bounded in $C^2(\mathbb{R})$, it follows that the specific free energy $\bar{\psi}^*_J(\hat{\beta})$ of $\nu_{J,\beta}$ is convex in $\hat{\beta}$ for large $J$ and its Hessian is uniformly bounded from above and from
below. Consequently, the Legendre-Fenchel transform $\bar{\psi}_J(\beta)$ of $\bar{\psi}^*_J(\hat{\beta})$ is
uniformly strictly convex. Moreover, the difference $\bar{\psi}_J(\beta) - \psi_J(\beta)$ can be interpreted as the
difference between the specific free energies of $\nu_{J,\beta}$ and its restriction to the hyperplane determined by
$Q_{1}x=\beta$ (which is the ``typical event''). Hence, we expect that this difference goes to zero as $J$ grows large. \medskip

To verify that this difference indeed converges to zero in the uniform $C^2$-topology, we first relate it through a Cram\'er-type
representation formula with the evaluation at $0\in \mathbb{R}^{L+1}$ of the density of the random variable
$J^{\frac{1}{2}}\,(Q_{1}x - \beta)$, given that $x$ is distributed according to $\nu_{J,\beta}$. This is done
in Lemma~\ref{cramer}. We then establish a kind of uniform $C^2$ local
central limit theorem assuring that this evaluated density is bounded from above and
from below uniformly in $\beta$ and that moreover, it is bounded in $C^2$ as a function of $\beta$. These estimates are stated in Proposition~\ref{estimates on g}
below and constitute the core of our proof. Then, combining the statements of Lemma~\ref{cramer} and Proposition~\ref{estimates on g} allow us to deduce Theorem~\ref{1block}. \\

Let us now take the first step in proving Theorem~\ref{1block}, namely the construction of the modified grand canonical ensembles~$\nu_{J,\beta}$. We begin by introducing a family of ``exponential shifts'' of the single-site measure. For each $m \in \mathbb{R}$,  let $\mu_m$ be the probability measure on~$\mathbb{R}$ given by the Lebesgue density
\begin{align}
  \label{e_d_mu_m}
  \frac{d \mu_m}{dz} =\exp(-\psi^*(\hat{z}_m) + \hat{z}_mz - \psi(z))
\end{align}
where $\psi^*$ is the log moment generating function for the single-site potential $\psi$ (cf. \eqref{psi^*}), and ~$\hat{z}_m$ is chosen so that $\mu_m$ has mean $m$, i.e.
\begin{align}
  \int_{\mathbb{R}} z  \mu_m(dz) = (\psi^* )'(\hat z_m) = m.
\end{align}
We will use the fact that the function $\psi^*$ is uniformly strictly convex and its 2nd and 3rd derivatives are uniformly bounded. 
\begin{lemma} \label{psi*}
There are $0<c<C<\infty$ such that it holds that:
\begin{align}
 0 < 4c < \inf_{m\in\mathbb{R}} \operatorname{Var}(\mu_m) \leq(\psi^* )''(\hat{z})
\leq \sup_{m\in\mathbb{R}} \operatorname{Var}(\mu_m) < \frac{C}{2}< \infty, \label{boundvar}\\
\left|(\psi^* )'''(\hat{z})\right| \leq \sup_{m\in\mathbb{R}} \left|\int (z-m)^3 \mu_m(dz) \right|<
 \frac{C}{2} < \infty. \label{boundthirdm}
\end{align}
where $\operatorname{Var}(\mu_m)$ denotes the variance of $\mu_m$.
\end{lemma}
We omit the proof of this result. It is contained in Lemma 41 in \cite{GORV}. \medskip

To construct $\nu_{j, \beta}$, we need to find a suitable dual variable $\hat{\beta}\in Y_1^{DG}$ and shift the $J$ spins
exponentially according to the $J$-dimensional vector
$JQ_1^t \hat \beta$ (which is the $L^2$ orthogonal projection of $\hat \beta$ onto $\mathbb{R}^J$, cf. Definition \ref{d_adjoint_NQMt}). To facilitate calculations from now on, we give an explicit coordinate representation of the operators $Q_1, JQ_1^t$: 

\begin{definition}\label{Q1_basis} 
	Fix once for all an orthonormal basis $\{f_l\}_{l=0,1,\cdots,L}$ on $Y_1^{DG}$ (which are a family of orthogonal polynomials of degree $\leq L$), and use this to identify $Y_1^{DG}$ with $\mathbb{R}^{L+1}$ (the $L^2$ structure on the former is identified with the standard Euclidean structure on the latter). \medskip
	
	Under this identification $Y_1^{DG} = \mathbb{R}^{L+1}$ and the standard identification $X_J = \mathbb{R}^J$, the operator $Q_1: X_J \rightarrow Y_1^{DG}$ is represented as a $(L+1) \times J$ matrix $\Gamma$, and the operator $JQ_1^t: Y_1^{DG} \rightarrow X_J$ is represented by the matrix $J\Gamma^t$, where $\Gamma^t$ denotes the transpose of $\Gamma$. \medskip
	 
	For $0\leq l\leq L$, $1\leq j\leq J$, let $\bar f_l \in \mathbb{R}^J = X_J$ and $\gamma^j \in \mathbb{R}^{L+1} = Y_1^{DG}$ be the $l$-th column and $j$-th row of the matrix $J\Gamma^t$, respectively, then
	\begin{align}
	(Q_1 x)_l = \frac{1}{J} x \cdot \bar f_l, \quad (JQ_{1}^t\,y)_j =  y\cdot \gamma^j.
	\end{align}
	
\end{definition}

Let $\bar{\psi}^*_J: \mathbb{R}^{L+1} \rightarrow \mathbb{R}$ be the function
\begin{align}
\bar{\psi}^*_J(\hat{\beta})\,:= \frac{1}{J} \,\sum_{j=1}^{J} \psi^*\left((JQ_1^t \hat \beta)_j\right) = \frac{1}{J} \,\sum_{j=1}^{J} \psi^*\left(\hat \beta \cdot \gamma^j \right). \label{barpsi*_def}
\end{align}
which will be interpreted as the specific free energy of $\nu_{j, \beta}$ for the right choice of $\hat \beta$. After dealing with the approximation error, property
\eqref{boundvar} translates into convexity bounds for the function $\bar{\psi}^*_J(\hat{\beta})$ :
\begin{lemma} \label{barpsi*}
Let $c$ and $C$ be as in Lemma \ref{psi*}. Then there is $J_1 \in \mathbb{N}$ such that for all $J\geq J_1$ and
$\hat{\beta},\hat{\eta},\hat{\gamma} \in \mathbb{R}^{L+1}$:
\begin{align}
2c\,\,\leq\, \|\operatorname{Hess} \,\bar{\psi}^*_J(\hat{\beta})\| \,&\leq\,
C\,, \label{bdhesspsiJ*}\\
\|\,D^3 \bar{\psi}^*_J (\hat{\beta})\,\| \,&\leq\,C \label{bdthirdpsiJ*}.
\end{align}
\end{lemma}
We postpone the proof of this result to Section~\ref{proof1block}.\medskip

By a standard result from convex analysis, the bounds ~\eqref{bdhesspsiJ*} and ~\eqref{bdthirdpsiJ*} imply
that for $J\geq J_1$, the Legendre-Fenchel transform of $\bar{\psi}^*_J$,
\begin{equation}
\bar{\psi}_J(\beta) := \sup_{\hat{\beta}\in \mathbb{R}^{L+1}} \left(\langle \beta,\hat{\beta}\rangle
                       \,-\,\bar{\psi}^*_J(\hat{\beta}) \right) \label{d_bdhesspsiJ*}
\end{equation}
is uniformly strictly convex and its Hessian and 3rd derivative are uniformly bounded, and the unique maximizer $\hat{\beta}^{max}(\beta)$ of~\eqref{d_bdhesspsiJ*} satisfies 
\begin{equation}
\beta \,=\, \nabla \,\bar{\psi}^*_J(\hat{\beta}^{max}), \quad \hat \beta^{max} \,=\, \nabla \,\bar{\psi}_J(\beta).\label{duality}
\end{equation}
The vector~$\hat \beta^{max}$ serves to construct $\nu_{J, \beta}$. For $j=1,...,J$, set
\begin{align}
   \hat{m}_{j, \beta} &:= \ (JQ_1^t \hat \beta^{max} )_j = \hat \beta^{max} \cdot \gamma^j, \label{dualmeanj} \\
   m_{j,\beta} &:=  
   (\psi^* )'(\hat{m}_{j,\beta}). \label{meanj}
\end{align}
and define a product measure on $ \mathbb{R}^J$ (cf. \eqref{e_d_mu_m})
\begin{align}
\frac{d\nu_{J,\beta}}{dx}(x) \,:=\,\prod_{j=1}^{J} \frac{d\mu_{m_{j, \beta}}}{dx_j}(x_j). \label{mgrand}
\end{align}


It remains to check that the expected value of $Q_{1}x$ under $\nu_{J,\beta}$ is equal to $\beta$. For $l=0,...,L$:
\begin{align}
 \int & (Q_{1}x)_l \, d\nu_{J,\beta} =
 \int \frac{1}{J}x \cdot \bar f_l \, d\nu_{J,\beta} \overset{\eqref{mgrand}}{=} \frac{1}{J} \sum_{j=1}^{J} m_{j,\beta} (\bar f_l)_j \\
 &\overset{\eqref{meanj}}{=} \frac{1}{J} \sum_{j=1}^{J} (\psi^* )'(  \hat m_{j, \beta}  ) \,(\gamma^j)_l \overset{\eqref{barpsi*_def}, \eqref{dualmeanj}}{=} \frac{\partial}{\partial\hat{\beta}}_l \bar{\psi}^*_J(\hat{\beta}^{max}) \overset{\eqref{duality}}{=}  \beta_l. 
\end{align}

Thus, the construction of the modified grand canonical ensemble $\nu_{J, \beta}$ is complete.\medskip

For a given $\beta$, the specific free energy of the modified grand canonical ensemble $\nu_{J, \beta}$ is
just $\bar{\psi}^*_J (\hat{\beta}^{max})\,=\, \langle \beta, \hat{\beta}^{max}\rangle \,-\, \bar{\psi}_J(\beta)$.
On the other hand, the specific free energy of the canonical ensemble associated with the restriction of $\nu_{J,\beta}$
to the hyperplane $\{x\,|\,Q_{1}x\,=\,\beta\}$, where it is highly concentrated anyway for large $J$, is given by $\langle \beta, \hat{\beta}^{max}\rangle \,-\, \psi_J (\beta)$ (we leave the calculation to the reader as an exercise). Consequently, $\bar{\psi}_J (\beta) - \psi_J(\beta)$ measures the difference in free energies and hence
we expect it to converge to zero in some sense as $J\rightarrow \infty$.
As we indicated above, the proof that it converges strong enough for our purposes, i.e. in $C^2(\mathbb{R}^{L+1})$,
begins with a Cram\'er-type representation formula for the density in $0\in \mathbb{R}^{L+1}$ of
the distribution of $J^{\frac{1}{2}}\,(Q_{1}x - \beta)$ under $\nu_{J,\beta}$, which is a centered
$(L+1)$-dimensional vector of ``suitably weighted"´ sums of independent random variables. (Cf. equation (125) in \cite{GORV})
\begin{lemma} \label{cramer}
Denote by $g_{J,\beta}$ the law of the $\mathbb{R}^{L+1}$-valued random
variable $J^{\frac{1}{2}}\,(Q_{1}x - \beta)$ under $\nu_{J,\beta}$ and let
$\mathcal{J}Q \,:=\, (\det Q_{1}Q_{1}^t)^{\frac{1}{2}}$. The density of
$g_{J,\beta}$ at $0\in\mathbb{R}^{L+1}$ with respect to Lebesgue measure can be represented as follows:
\begin{equation}
\frac{dg_{J,\beta}}{d\mathcal{L}^{L+1}}(0) \,=\, (J^{\frac{L+1}{2}} \,\mathcal{J}Q)^{-1}\;
\exp\,[J\,(\bar{\psi}_J (\beta) - \psi_J(\beta))]. \label{cramerrep}
\end{equation}
\end{lemma}
We postpone the proof of Lemma~\ref{cramer} to Section~\ref{proof1block}. Formula ~\eqref{cramerrep} allows us to transfer the strict convexity of~$\bar{\psi}_J$ to the function~$\psi_J$, once we have the following estimates on the Jacobian~$(J^{\frac{L+1}{2}} \,\mathcal{J}Q)^{-1}$ (appearing on the right hand side of~\eqref{cramerrep}) and the density~$\frac{dg_{J,\beta}}{d\mathcal{L}^{L+1}}(0)$ (appearing on left hand side of~\eqref{cramerrep}).

\begin{lemma}\label{p_estimate_jacobien}
  There is a positive integer $J^* \in \mathbb{N}$ such that for $J\,\geq\,J^*$:
\begin{equation}
\frac{1}{C} \,\leq\, J^{\frac{L+1}{2}} \,\mathcal{J}Q \,\leq\, C. \label{Jacobian}
\end{equation}
\end{lemma}
We postpone the proof of Lemma~\ref{p_estimate_jacobien} to Section~\ref{proof1block}.
\begin{proposition}\label{estimates on g}
There exist a constant $C \,<\, \infty$ and a positive integer $J_2$ such that for all $J\geq J_2$ and all $\beta \in \mathbb{R}^{L+1}$:
\begin{align}
   \frac{1}{C} \,\leq\, \frac{dg_{J,\beta}}{d\mathcal{L}^{L+1}}(0) \,\leq\, C,& \label{boundg} \\
  \left\|\nabla\,\frac{dg_{J,\beta}}{d\mathcal{L}^{L+1}}(0)\right\| \,\leq\, C,& \label{boundnablag}\\
  \left\|\operatorname{Hess}\, \frac{dg_{J,\beta}}{d\mathcal{L}^{L+1}}(0)\right\| \,\leq\, C&. \label{boundhessg}
\end{align}
\end{proposition}
This result was proven in~\cite{GORV} for the case $L=0$ (cf. equation (126) in ~\cite{GORV}). In the general case considered here,
establishing the estimates becomes somewhat more subtle. In particular, a geometric property due to
the independence of basis polynomials $f_l$ enters as a new ingredient. The proof also shares some similarities to the proof of the local Cram\'er theorem in~\cite{Menz11}.  As the proof as a whole becomes
quite long we postpone it to Section~\ref{s_p_local_cramer}.
We conclude this Section with a derivation of Theorem~\ref{1block} from these results.

\begin{proof}[Proof of Theorem~\ref{1block}]
Rewrite formula \eqref{cramerrep} as:
\begin{equation}
\bar{\psi}_J (\beta) - \psi_J(\beta) \,=\,\frac{1}{J}\, \left[\ln \,(J^{\frac{L+1}{2}} \,\mathcal{J}Q)\,+\,\ln\,
\frac{dg_{J,\beta}}{d\mathcal{L}^{L+1}}(0)\right]. \label{cramerrep2}
\end{equation}
For $J\,\geq\,\max\{J^*, J_2\}$, the estimates and \eqref{boundg} and~\eqref{Jacobian} thus yield
\begin{equation}
|\bar{\psi}_J (\beta) - \psi_J (\beta)| \,\leq\, \frac{\ln C+\ln C}{J}. \label{orderzero}
\end{equation}
For the gradient of the difference we find:
\begin{align}
\|(\nabla \bar{\psi}_J - \nabla \psi_J)(\beta)\|  &\stackrel{\eqref{cramerrep2}}{=} 
\frac{1}{J} \,\left(\,\frac{dg_{J,\beta}}{d\mathcal{L}^{L+1}}(0)\,\right)^{-1}\;
\left\|\nabla\frac{dg_{J,\beta}}{d\mathcal{L}^{L+1}}(0)\right\| \\
  & \stackrel{\eqref{boundg},\eqref{boundnablag}}{\leq}\, \frac{C^2}{J}. \label{e_local_cramer_convergence_gradient}
\end{align}
Finally,
\begin{align}
& \|(\operatorname{Hess} \,\bar{\psi}_J - \operatorname{Hess} \,\psi_J)(\beta)\| \\ 
& \qquad \stackrel{\eqref{cramerrep2} }{\leq}  \frac{1}{J} \,\left(\,\frac{dg_{J,\beta}}{d\mathcal{L}^{L+1}}(0)\,\right)^{-1}
 \,\left\|\operatorname{Hess} \,\frac{dg_{J,\beta}}{d\mathcal{L}^{L+1}}(0)\right\| \\
 & \qquad \qquad \qquad  + \frac{1}{J} \,\left(\,\frac{dg_{J,\beta}}{d\mathcal{L}^{L+1}}(0)\,\right)^{-2} \, \left\|\,\nabla\frac{dg_{J,\beta}}{d\mathcal{L}^{L+1}}(0)\,\otimes\,\nabla\frac{dg_{J,\beta}}{d\mathcal{L}^{L+1}}(0)\,\right\|\\
&\qquad \stackrel{\eqref{boundg},\eqref{boundnablag},\eqref{boundhessg}}{\leq}  \frac{C^2}{J} \,+\, \frac{C^4}{J}. \label{e_local_cramer_convergence_hessian}
\end{align}
This proves convergence of $\psi_J$ to $\bar{\psi}_J$ in $C^2(\mathbb{R}^{L+1})$, indeed with difference of order
$J^{-1}$ as $J\rightarrow\infty$. Since $\bar{\psi}_J$ is
uniformly strictly convex and its Hessian is uniformly bounded if $J\geq J_1$, this 
proves Theorem~\ref{1block}.
\end{proof}

\subsection{Proofs of the auxiliary results of Section~\ref{s_strict_convexity_bar_H_DG}.} \label{proof1block}

\begin{proof}[Proof of Lemma \ref{barpsi*}.] For $\hat{\eta}\in \mathbb{R}^{L+1}$ we have:
\begin{eqnarray*}
\langle \hat{\eta}, \operatorname{Hess} \,\bar{\psi}^*_J(\hat\beta) \,\hat{\eta}\rangle
&=& \frac{1}{J} \,\sum_{j=1}^J (\psi^* )''\left((JQ_1^t \hat \beta)_j\right) \,
\left[(JQ_1^t \hat \eta)_j\right]^2.
\end{eqnarray*}
Using \eqref{boundvar} we obtain:
\begin{equation*}
\frac{4c}{J}  \,\sum_{j=1}^J \left((JQ_1^t \hat \eta)_j\right)^2
\leq\,\langle \hat{\eta}, \operatorname{Hess} \,\bar{\psi}^*_J(\hat{\beta}) \,\hat{\eta}\rangle
\leq\, \frac{C}{2J}  \,\sum_{j=1}^J \left((JQ_1^t \hat \eta)_j\right)^2.
\end{equation*}
But the approximation of Lemma \ref{PNP} below implies:
\begin{align}\frac{1}{J} &\,\sum_{j=1}^J \left((JQ_1^t \hat \eta)_j\right)^2 = \frac{1}{J} \, |JQ_1^t \hat \eta|^2
= \langle JQ_1^t \hat \eta, JQ_1^t \hat \eta\rangle_{L^2} \\
&= \langle Q_1JQ_1^t \hat \eta,  \hat \eta\rangle_{L^2} 
\stackrel{\eqref{approxid2}}{=}  \left(1 + O\left(\frac{1}{J^2}\right)\right) |\hat \eta|^2.
\end{align}
This shows \eqref{bdhesspsiJ*} for large $J$. Concerning \eqref{bdthirdpsiJ*}, we find for $\hat{\eta}\in \mathbb{R}^{L+1}$:
\begin{align}|\,D^3 \bar{\psi}^*_J (\hat{\beta}) &(\hat{\eta},\hat{\eta},\hat{\eta})\,| = \left|\frac{1}{J} \,\sum_{j=1}^J (\psi^* )'''
\left((JQ_1^t \hat \beta)_j  \right) \:
\left[(JQ_1^t \hat \eta)_j\right]^3 \right|\\
&\leq \max_{j=1,...,J} \left|\hat \eta \cdot \gamma^j\right|\frac{1}{J} \sum_{j=1}^J \left|(\psi^* )'''
\left((JQ_1^t \hat \beta)_j \right) \right|\:
\left[(JQ_1^t \hat \eta)_j\right]^2 
\end{align}
We then appeal to \eqref{boundthirdm} and the uniform bound \eqref{gammaj_approx} from below and proceed just as above. 
\end{proof}

\begin{proof}[Proof of Lemma \ref{cramer}.]

So far it has always been understood that by $dx$, etc. we mean the Hausdorff measure of appropriate dimension. We will be a bit more careful during the next computation and write out the measures in detail where it seems helpful. Let $\zeta$ be a measurable test function defined on $\mathbb{R}^{L+1}$. The proof of identity
\eqref{cramerrep} essentially boils down to an application of the co-area Formula
for $Q_{1}$.
\begin{eqnarray*}
  \lefteqn{\int_{\mathbb{R}^{L+1}} \zeta(u) \,\frac{dg_{J,\beta}}{d\mathcal{L}^{L+1}}(u) \, \mathcal{L}^{L+1} (du) = \int_{\mathbb{R}^J} \zeta\left(J^{\frac{1}{2}}(Q_{1}x - \beta)\right)\, \frac{d\nu_{J,\beta}}{d\mathcal{L}^J}(x) \,\mathcal{L}^J(dx)}\\
&=&\int_{\mathbb{R}^{L+1}} (\mathcal{J}Q)^{-1} \zeta\left(J^{\frac{1}{2}}(\widetilde{\beta}-\beta)\right) 
  \int_{\mathbb{R}^J_{1,\widetilde{\beta}}} \frac{d\nu_{J,\beta}}{d\mathcal{L}^{J}}(x) \mathcal{H}^{J-L-1}(dx)
   \mathcal{L}^{L+1}(d\widetilde{\beta}) \\
&\stackrel{\eqref{mgrand}}{=}&\int_{\mathbb{R}^{L+1}} (\mathcal{J}Q)^{-1}\zeta\left(J^{\frac{1}{2}}(\widetilde{\beta}-\beta)\right) \\
  && \int_{\mathbb{R}^J_{1,\widetilde{\beta}}}\exp\left(-\sum_{j=1}^J\psi^*(\hat{m}_{j,\beta})+
  \sum_{j=1}^J\hat{m}_{j,\beta} x_j - H_J(x)\right) \mathcal{H}(dx)  \mathcal{L}(d\widetilde{\beta})\\
&\stackrel{\eqref{barpsi*_def},\eqref{dualmeanj}}{=}&\int_{\mathbb{R}^{L+1}} \,(\mathcal{J}Q)^{-1}\,\zeta\left(J^{\frac{1}{2}}(\widetilde{\beta}-\beta)\right) \\
  && \;\int_{\mathbb{R}^J_{1,\widetilde{\beta}}}\,  \exp\left(-J\bar{\psi}^*_J (\hat{\beta}^{max}) + J\, \langle\widetilde{\beta},\hat{\beta}^{max}\rangle
  - H_J(x)\right)\mathcal{H}(dx)  \mathcal{L}(d\widetilde{\beta})\\
&\stackrel{\eqref{d_bdhesspsiJ*}, \eqref{psiJ}}{=}& \int_{\mathbb{R}^{L+1}} \,(\mathcal{J}Q)^{-1}\,\zeta\left(J^{\frac{1}{2}}(\widetilde{\beta}-\beta)\right)\\
 && \quad \exp\left(J\left(\bar{\psi}_J (\beta)- \psi_J (\widetilde{\beta})\,+ \langle\widetilde{\beta} - \beta,\hat{\beta}^{max}\rangle\right)\right)
  \mathcal{L}^{L+1}(d\widetilde{\beta})\\
&=& \int_{\mathbb{R}^{L+1}} \zeta(u)\left(J^{\frac{L+1}{2}}\mathcal{J}Q\right)^{-1}\\
 && \quad \exp \left(J\left(\bar{\psi}_J (\beta)- \psi_J (J^{-\frac{1}{2}}u+\beta) +\langle J^{-\frac{1}{2}}u,\hat{\beta}^{max}\rangle\right)\right)
 \mathcal{L}^{L+1}(du).
\end{eqnarray*}
The identity \eqref{cramerrep} now follows from approximating the Dirac
mass $\delta_0$ in $\mathbb{R}^{L+1}$ by continuous test functions $\zeta_i$.
\end{proof}

Lemma~\ref{p_estimate_jacobien} is a simple consequence of the following statement which says that the operator~$Q_{1}JQ_{1}^t$ is close to the identity for large~$J$. 
\begin{lemma} \label{PNP}
It holds (with implicit constants depending on $L$):
\begin{align}
(Q_{1}JQ_{1}^t)_{l_1 l_2} &=\, \delta_{l_1 l_2} + O\left(\frac{1}{J^2}\right),  \label{approxid2} \\
\|Q_{1}JQ_{1}^t -\operatorname{id}_{Y_{1}^{DG}}\| &\lesssim  \frac{1}{J^2}. \label{approxid3}
\end{align}
\end{lemma}

\begin{proof}[Proof of Lemma~\ref{PNP}.]
By Definitions \ref{d_adjoint_NQMt} and \ref{Q1_basis}, for $l_1,l_2= 0,...,L$:
\begin{align}
(\bar f_l)_j &= (JQ_1^t f_l)_j = J \int_{\frac{j-1}{J}}^{\frac{j}{J}} f_l d\theta, \\
(Q_{1}JQ_{1}^t f_{l_1})_{l_2} &= (Q_1 \bar f_{l_1})_{l_2} = \frac{1}{J}\bar f_{l_1} \cdot \bar f_{l_2} = \langle \bar f_{l_1}, \bar f_{l_2} \rangle_{L^2}. \label{approxid1}
\end{align}
Since $\langle f_{l_1}, f_{l_2} \rangle_{L^2} = \delta_{l_1 l_2}$ and $\bar f_{l}$ is a piecewise average approximation of $f_l$, we expect
$Q_{1}JQ_{1}^t$ converges to $\id_{Y^{DG}}$. More precisely, 
\begin{align}\langle  f_{l_1}, f_{l_2} \rangle_{L^2} - \langle \bar f_{l_1}, \bar f_{l_2} \rangle_{L^2}  &= \langle f_{l_1} - \bar f_{l_1}, f_{l_2} - \bar f_{l_2} \rangle_{L^2}\\ 
&\leq |f_{l_1} - \bar f_{l_1}|_{L^\infty} |f_{l_2} - \bar f_{l_2}|_{L^\infty} \\
&\leq \left(\frac{1}{J} |f_{l_1}'|_{L^{\infty}}\right) \left(  \frac{1}{J} |f_{l_2}'|_{L^{\infty}} \right)  
\end{align}
where we used mean value theorem. This implies \eqref{approxid2}, and \eqref{approxid3} follows. \end{proof}

\begin{proof}[Proof of Lemma~\ref{p_estimate_jacobien}]
We note that $J^{\frac{L+1}{2}} \,\mathcal{J}Q \,=\, (\det Q_{1}JQ_{1}^t)^{\frac{1}{2}}$.
Estimate ~\eqref{approxid3} implies that there is
$J^* \in \mathbb{N}$ such that for $J\,\geq\,J^*$:
\begin{equation}
\frac{1}{C} \,\leq\, J^{\frac{L+1}{2}} \,\mathcal{J}Q \,\leq\, C.
\end{equation}
\end{proof}

\subsubsection{Proof of Proposition \ref{estimates on g}.}\label{s_p_local_cramer}

We now begin with the rather long and technical proof of Proposition \ref{estimates on g}. We recommend the interested reader to first read the proof of Proposition 31 in \cite{GORV}. As in the usual proof of the (local) central limit theorem, we use independence and the Fourier transform to
obtain an explicit formula for $\frac{dg_{J,\beta}}{d\mathcal{L}^{L+1}}(0)$. This is the starting point of our further
analysis.
\begin{lemma}\label{gexplicit}
Let
\begin{equation}
 h(m,z)\,:=\, e^{-imz} \int_{\mathbb{R}} e^{iz x} \,\mu_m(dx) \label{h}
\end{equation}
be the characteristic function of the centered version of the measure $\mu_m$ given by~\eqref{e_d_mu_m}. Then $\frac{dg_{J,\beta}}{d\mathcal{L}^{L+1}}(0)$ can be represented as
\begin{equation}
 \frac{dg_{J,\beta}}{d\mathcal{L}^{L+1}}(0)
 \,=\,\left(\frac{1}{2\pi}\right)^{L+1}\,\int_{\mathbb{R}^{L+1}} \, \prod_{j=1}^{J} \,h(m_{j,\beta}\,,\,
 J^{-\frac{1}{2}}\,\xi\cdot\gamma^j) \,d\xi. \label{repg}
\end{equation}
\end{lemma}

\begin{proof}[Proof of Lemma \ref{gexplicit}.]
Applying Fourier transform,
\begin{align} (2\pi)&^{L+1} \lefteqn{\frac{dg_{J,\beta}}{d\mathcal{L}^{L+1}}(0) \,=\,  \int_{\mathbb{R}^{L+1}} \, \int_{\mathbb{R}^{L+1}}
  \exp\,(i\,\xi\cdot u)\,\frac{dg_{J,\beta}}{d\mathcal{L}^{L+1}}(u)\, du \,d\xi}\\
 &= \int_{\mathbb{R}^{L+1}} \, \int_{\mathbb{R}^J} \exp\left(i\,\xi\cdot\,J^{\frac{1}{2}}(Q_{1}x-\beta)\right)
 \nu_{J,\beta}(dx) \,d\xi\\
 &= \int_{\mathbb{R}^{L+1}} \, \int_{\mathbb{R}^J} \exp\left(i\,\xi\cdot\,J^{\frac{1}{2}}(Q_{1}x-\int Q_1 \tilde{x} \nu_{J, \beta}(d\tilde{x}))\right)
 \nu_{J,\beta}(dx) \,d\xi\\
 &\stackrel{\eqref{mgrand}}{=}  \int_{\mathbb{R}^{L+1}} \, \prod_{j=1}^{J}\, \int_{\mathbb{R}}
      \exp \left(i\, J^{-\frac{1}{2}}\, (JQ_1^t \xi)_j \, (x_j - m_{j,\beta})\right)\, \mu_{m_{j,\beta}}(dx_j)\, d\xi \\
&\stackrel{\eqref{h}}{=} \int_{\mathbb{R}^{L+1}} \,
  \prod_{j=1}^{J}\, h(m_{j,\beta}\,,\,J^{-\frac{1}{2}}\,\xi\cdot\gamma^{j}) \,d\xi.
\end{align}
as desired.
\end{proof}
To continue from formula \eqref{repg} we need two ingredients. The first ingredient is a collection of elementary
properties of the function $h$. 
\begin{lemma} \label{properties of h}
We have the following bounds and decay properties for the function $h$ and its derivatives:
\begin{equation}
|h(m,z)| \leq 1. \label{globalboundh}
\end{equation}
Given $\varepsilon > 0$, there is $C_{\varepsilon}<\infty$ (uniform in $m$), such that for $|z|\geq \varepsilon$:
\begin{equation}
|h(m,z)| \,\leq\, \frac{1}{1 + |z|\,C_{\varepsilon}^{-1}}. \label{decayh}
\end{equation}
For all $z \in \mathbb{R}$, $m\in\mathbb{R}$:
\begin{equation}
\left|\frac{\partial h}{\partial m} (m,z)\right| \leq C (1+ |z|), \quad
\left|\frac{\partial^2 h}{\partial m^2} (m,z)\right| \leq C (1+ |z|^2).  \label{secderh}
\end{equation}
There is $\delta_0> 0$ such that for $z \in [-\delta_0,\delta_0]$, $m\in\mathbb{R}$, we can express $h$ as
\begin{equation}
h(m,z) = \exp(-z^2\,h_2(m,z)). \label{h2}
\end{equation}
Here, $h_2$ is a (complex-valued) function satisfying
\begin{align}
 h_2(m,0) = \frac{\operatorname{Var}(\mu_m)}{2}, \ 0 <  c \leq \text{Re } h_2(m,z) \leq C \ \forall z\in [-\delta_0,\delta_0], \label{h2zero} \\
\left|\frac{\partial h_2}{\partial z} (m,z)\right| \leq C ,\quad
\left|\frac{\partial h_2}{\partial m} (m,z)\right| \leq C ,\quad \left|\frac{\partial^2 h_2}{\partial m^2} (m,z)\right| \leq C .\label{derhinside}
\end{align}
\end{lemma}
The estimates of Lemma~\eqref{properties of h} should not be surprising as
$h(m,\cdot)$ is just the Fourier transform of $\mu_m$ which belongs to the exponential family of a perturbed standard Gaussian measure. For the proofs, we refer
the reader to~\cite{GORV}.\medskip

The second ingredient for the proof of Proposition~\ref{estimates on g} is a lower bound on the inner products $\xi \cdot \gamma^j$ which enter into the second argument of $h$. This is new compared to~\cite{GORV}.
Note that the vectors $\{\gamma^j\}_{j=1}^J$ form a piecewise constant approximation of the smooth curve
\begin{equation}
\gamma\,: [0,1] \rightarrow \mathbb{R}^{L+1}\,,\: \gamma(t) \,=\,(f_0(t), f_1(t), \cdots, f_L(t)).
\end{equation}
in the sense that $\forall J \in \mathbb{N}$ (with implicit constants depending on $L$),
	\begin{align}
	\max_{1\leq j\leq J} \sup_{t\in [\frac{j-1}{J}, \frac{j}{J}]} |\gamma^j - \gamma(t)| \lesssim \frac{1}{J} \quad \mbox{and} \quad
	\max_{1\leq j\leq J} |\gamma^j| \lesssim 1, \label{gammaj_approx}
	\end{align}
The proof of \eqref{gammaj_approx} is similar to that of \eqref{approxid2}. It turns out that all we need is the following elementary geometric property of the curve $\gamma$ and its piecewise approximation $\{\gamma^j\}_{j=1}^J$.
\begin{lemma} \label{gamma} Fix $L+1$ disjoint closed subintervals of $[0, 1]$ of length $1/(L+1.5)$, and denote them $I_k$, $1\leq k \leq L+1$. For $\xi\in\mathbb{R}^{L+1}$, define
\begin{equation*}
\omega_k(\xi)\,:=\, \inf_{t\in I_k} |\xi\cdot\gamma(t)|\, \quad \mbox{ and } \quad \omega_{k,J}(\xi)\,:=\, \min_{j: \frac{j}{J}\in I_k} |\xi\cdot\gamma^j|
\end{equation*}
There is a constant $c_{\gamma} > 0$ such that
\begin{equation}
\inf_{\xi \in S^L} \,\max_{k=1,...,L+2}\, \omega_k(\xi) \,\geq\, 2\,c_{\gamma} \label{gammacont}
\end{equation}
where $S^L$ denotes the unit sphere in $\mathbb{R}^{L+1}$. From \eqref{gammacont}, it follows by approximation (cf. \eqref{gammaj_approx}) that there exists an integer $J_{\gamma}$ such that for $J\geq J_{\gamma}$, each interval $I_k$ contains at least $J/(L+2)$ spins and 
\begin{align}
 \,\max_{k=1,...,L+2}\, \omega_{k,J}(\xi) \geq c_{\gamma}. \label{gammadisc}
\end{align}
\end{lemma}
Denote by $k(\xi)$ the index of the (first) interval $I_k$ that maximizes $\omega_{k, J}$. 
\begin{proof}[Proof of Lemma~\ref{gamma}.]

The function $\omega_k:\mathbb{R}^{L+1}\rightarrow \mathbb{R}$ is continuous for all $k=1,...,L+1$, so the same is true for
$\omega :=\max_k \omega_k$.
As $S^L$ is compact, for (\ref{gammacont}) it only remains to show that
\begin{equation*}
\forall \xi \in S^L:\, \omega(\xi) > 0.
\end{equation*}
But for $\xi \in S^L$, $\xi \cdot \gamma$ is, by independence of the polynomials $f_l$, a polynomial of degree $\leq L$ that is not
identically zero. Hence, it has at most $L$ zeros in $[0,1]$, which implies $\omega(\xi) > 0$ by pigeonhole principle. \end{proof}

The strategy for the rest of the proof is to split the integral on the right hand side of \eqref{repg},
i.e. $$(2\pi)^{L+1}\,\frac{dg_{J,\beta}}{d\mathcal{L}^{L+1}}(0)
\,=\,\int_{\mathbb{R}^{L+1}} \, \prod_{j=1}^{J} \,h(m_{j,\beta}\,,\,
J^{-\frac{1}{2}}\,\xi\cdot\gamma^j) \,d\xi,$$
 into an inner and an outer part. 
We will show that for sufficiently small $\delta$ and for sufficiently large $J$ (depending on $\delta$)
\begin{align}
&  \lim_{J \to \infty}\int_{\left\{|\xi|> J^{\frac{1}{2}}\delta\right\}} \, \prod_{j=1}^{J} |\,h(m_{j,\beta}\,,\,J^{-\frac{1}{2}}\,\xi\cdot\gamma^{j})| \;d\xi = 0 , \label{outerint} \\
&\int_{\left\{|\xi|\leq J^{\frac{1}{2}}\delta\right\}} \, \prod_{j=1}^{J} |\,h(m_{j,\beta}\,,\,J^{-\frac{1}{2}}\,\xi\cdot\gamma^{j})| \;d\xi
 \;\leq\; C , \label{innerintup}\\
& \left|\,\int_{\left\{|\xi|\leq J^{\frac{1}{2}}\delta\right\}} \, \prod_{j=1}^{J} \,h(m_{j,\beta}\,,\,J^{-\frac{1}{2}}\,\xi\cdot\gamma^{j}) \;d\xi \right|
 \;\geq\; \frac{1}{C} \label{innerintlow}\\
 & \lim_{J \to \infty} \left\|\,\operatorname{Hess}\, \int_{\left\{|\xi|> J^{\frac{1}{2}}\delta\right\}} \,\prod_{j=1}^{J} \,h(m_{j,\beta}\,,\,J^{-\frac{1}{2}}\,\xi\cdot\gamma^{j})\, \,d\xi \;\right\| = 0, \label{outerhess}\\
 &\left\|\,\operatorname{Hess}\,\int_{\left\{|\xi|\leq J^{\frac{1}{2}}\delta\right\}} \, \prod_{j=1}^{J} \,h(m_{j,\beta}\,,\,J^{-\frac{1}{2}}\,\xi\cdot\gamma^{j})\, \,d\xi
   \;\right\|\;\leq\; C.\label{innerhess}
\end{align}
The bounds for $\frac{dg_{J,\beta}}{d\mathcal{L}^{L+1}}(0)$ in \eqref{boundg} follows 
from \eqref{outerint} - \eqref{innerintlow}. The bounds for the Hessian in \eqref{boundhessg} follows from
\eqref{outerhess} and \eqref{innerhess}.
The bounds for the gradient in \eqref{boundnablag} is then immediate from interpolation.\medskip

Let us assume for the rest of the proof that $J\geq J_{\gamma}$. First consider
the outer integral from \eqref{outerint}. Recall that the interval $I_{k(\xi)}$ contains at least $J/(L+2)$ spins (cf. Lemma \ref{gamma}). With this and the decay property \eqref{decayh} of
$h$ in mind, we set $\varepsilon \,:=\, \delta\, c_{\gamma}$ and compute:
\begin{eqnarray*}
 \lefteqn{\int_{\left\{|\xi|\,>\, J^{\frac{1}{2}}\delta\right\}} \, \prod_{j=1}^{J} |\,h(m_{j,\beta}\,,\,J^{-\frac{1}{2}}\,\xi\cdot\gamma^{j})| \,d\xi}\\
&\stackrel{\eqref{globalboundh}}{\leq}& \int_{\left\{|\xi|\,>\, J^{\frac{1}{2}}\delta\right\}} \, \prod_{j: \frac{j}{J} \in I_{k(\xi)}}
   |\,h(m_{j,\beta}\,,\,J^{-\frac{1}{2}}\,\xi\cdot\gamma^{j})| \,d\xi\\
&\stackrel{\eqref{decayh}, \eqref{gammadisc}}{\leq}& \int_{\left\{|\xi|\,>\, J^{\frac{1}{2}}\delta\right\}} \, \prod_{j: \frac{j}{J} \in I_{k(\xi)}}
 \frac{1}{1+J^{-\frac{1}{2}} \,|\xi\cdot\gamma^{j}|\,C_{\varepsilon}^{-1}} \,d\xi\\
&\stackrel{\eqref{gammadisc}}{\leq}& \left(\frac{1}{1+\,\varepsilon\,C_{\varepsilon}^{-1}}\right)^{\frac{J}{L+2}-L-2}\;
       J^{\frac{L+1}{2}}\,\int_{\{|\xi|\,>\, \delta\}} \left(\frac{1}{1+\,c_{\gamma} C_{\varepsilon}^{-1} |\xi|}\right)^{L+2} \, d\xi.  
\end{eqnarray*}
This goes to $0$ as $J \rightarrow \infty$. \\

For the inner integral from \eqref{innerintup} we use the representation of $h$ via $h_2$ in Lemma \ref{properties of h}. For this purpose, we assume from now on that
\begin{equation}
\delta \,\leq\, \delta_0 \quad \mbox{and} \quad \delta \,\max_{j=1,...,J} |\gamma^j| \leq\, \delta_0. \label{delta0}
\end{equation}
We compute :
\begin{eqnarray*}
\lefteqn{\int_{\left\{|\xi|\,\leq\, J^{\frac{1}{2}}\delta\right\}} \, \prod_{j=1}^{J} |\,h(m_{j,\beta}\,,\,J^{-\frac{1}{2}}\,\xi\cdot\gamma^{j})| \;d\xi}\\
&\stackrel{\eqref{h2},\eqref{delta0}}{=}& \int_{\left\{|\xi|\,\leq\, J^{\frac{1}{2}}\delta\right\}} \, \prod_{j=1}^{J}\,
     |\exp\,(-\,(J^{-\frac{1}{2}}\,\xi\cdot\gamma^{j})^2 \; h_2(m_{j,\beta}\,,\,J^{-\frac{1}{2}}\,\xi\cdot\gamma^{j}))| \;d\xi\\
&\stackrel{\eqref{gammadisc},\eqref{globalboundh},\eqref{h2zero}}{\leq}& \int_{\left\{|\xi|\,\leq\, J^{\frac{1}{2}}\delta\right\}}\,
 \prod_{\left\{j\,|\,\frac{j}{J} \,\in\, I_{k(\xi)}\right\}}
    \,\exp \left(-c\,(J^{-\frac{1}{2}}\,c_{\gamma} |\xi|)^2\right) \;d\xi\\
&=& \int_{\left\{|\xi|\,\leq\, J^{\frac{1}{2}}\delta\right\}} \, \exp\left(-\,c \,J^{-1}\,\frac{J}{L+2}\, c_{\gamma}^2\, |\xi|^2\right) \;d\xi \\
&\leq& \int_{\mathbb{R}^{L+1}} \,\exp\left(-\,c \,\frac{1}{L+2} \,c_{\gamma}^2 \,|\xi|^2\right) \; d\xi,
\end{eqnarray*}
which is finite. Note that here we really need that the number of spins in $I_{k(\xi)}$ is of order $J$, whereas for the previous estimate arbitrarily slow increase to infinity of this number would have been sufficient. We next turn to the lower bound \eqref{innerintlow}:
\begin{eqnarray*}
\lefteqn{\left|\, \int_{\left\{|\xi|\leq J^{\frac{1}{2}}\delta\right\}} \,
\prod_{j=1}^{J} \,h(m_{j,\beta},J^{-\frac{1}{2}}\,\xi\cdot\gamma^{j}) \;d\xi\,\right|}\nonumber\\
&=& \left|\,\int_{\left\{|\xi|\leq J^{\frac{1}{2}}\delta\right\}} \,\exp\left(-\sum_{j=1}^J\,(J^{-\frac{1}{2}}\,\xi\cdot\gamma^{j})^2
     \; h_2(m_{j,\beta},J^{-\frac{1}{2}}\,\xi\cdot\gamma^{j})\right) \;d\xi\,\right|\nonumber\\
&\stackrel{\eqref{h2zero}}{\geq}& \int_{\left\{|\xi|\leq J^{\frac{1}{2}}\delta\right\}}
  \underbrace{\exp\left(- \sum_{j=1}^J\,(J^{-\frac{1}{2}}\,\xi\cdot\gamma^{j})^2
     \; h_2(m_{j,\beta},0)\right)}_{S_1} \;d\xi \nonumber\\
&-& \,\int_{\left\{|\xi|\leq J^{\frac{1}{2}}\delta\right\}} \exp\left(- \sum_{j=1}^J\,(J^{-\frac{1}{2}}\,\xi\cdot\gamma^{j})^2
\; h_2(m_{j,\beta},0)\right) \nonumber\\
&& \quad \underbrace{\left|\, \exp\left(- \sum_{j=1}^J\,(J^{-\frac{1}{2}}\,\xi\cdot\gamma^{j})^2
     \; [h_2(m_{j,\beta},J^{-\frac{1}{2}}\,\xi\cdot\gamma^{j}) - h_2(m_{j,\beta},0)]\right)-1\right|}_{S_2}\;d\xi.\label{lowerestimate}
\end{eqnarray*}
We estimate the terms $S_1, S_2$ as (using $|\exp(z) - 1| \leq \exp(|z|) - 1$):
\begin{align}
\exp\left(-C\,\max_j |\gamma^j|^2  |\xi|^2\right) \overset{\eqref{h2zero}}\leq S_1 &\overset{\eqref{h2zero},  \eqref{gammadisc}}{\leq} \exp\left(-c \frac{1}{L+2}\,c_{\gamma}^2\,|\xi|^2\right), \\
S_2 &\stackrel{\eqref{derhinside},\eqref{delta0}}{\leq} \exp\,(|\xi|^2 \max_j |\gamma^j|^2 C\delta_0) -1.
\end{align}
Set $C_1\,:=\,c \frac{1}{L+2}\,c_{\gamma}^2$ and $C_2\,:=\,\max_j |\gamma^j|^2 C$.
We thus find:
\begin{align}
& \int_{\left\{|\xi|\leq J^{\frac{1}{2}}\delta\right\}}  S_1 d\xi - \int_{\left\{|\xi|\leq J^{\frac{1}{2}}\delta\right\}} S_1 S_2 d\xi \\
&\qquad \geq \int_{\left\{|\xi|\leq J^{\frac{1}{2}}\delta\right\}}
  e^{-C_2 |\xi|^2} \;d\xi  - 
  \int_{\mathbb{R}^{L+1}}
  e^{-C_1 |\xi|^2} \left(e^{C_2 \delta_0 |\xi|^2} - 1\right)\;d\xi.
\end{align}
Now, we choose $\delta_0$ (and accordingly $\delta$) small enough to ensure that the first integral dominates the second integral for large $J$. This implies \eqref{innerintlow}
for all sufficiently large $J$.\medskip

Let us turn to the terms that involve derivatives with respect to $\beta$ now.
We first compute $\nabla\,\frac{dg_{J,\beta}}{d\mathcal{L}^{L+1}}(0)$
and $\operatorname{Hess}\, \frac{dg_{J,\beta}}{d\mathcal{L}^{L+1}}(0)$ starting from \eqref{repg}.
The interchange of differentiation and integration will be justified by the bounds
developed below, relying on pointwise bounds for the integrands. From now on, we write $[j] := (m_{j, \beta}, J^{-\frac{1}{2}} \xi \cdot \gamma^j)$ for short. We have:
\begin{equation}
(2\,\pi)^{L+1}\,\nabla \frac{dg_{J,\beta}}{d\mathcal{L}^{L+1}}(0) = \int_{\mathbb{R}^{L+1}} \sum_{j=1}^J \,\frac{\partial h}{\partial m}[j]\,
  \prod_{n\neq j} h[j]\; \nabla m_{j,\beta}\;d\xi
\end{equation}
and
\begin{align}
&(2\,\pi)^{L+1} \,\operatorname{Hess} \frac{dg_{J,\beta}}{d\mathcal{L}^{L+1}}(0) \label{Hessg}\\
 &= \int_{\mathbb{R}^{L+1}} \sum_{j=1}^J \,\frac{\partial^2 h}{\partial m^2}[j]\,
   \prod_{n\neq j} h[n]\;\nabla m_{j,\beta} \otimes \nabla m_{j,\beta}\,d\xi \nonumber\\  
   &+ \int_{\mathbb{R}^{L+1}} \sum_{j=1}^J \sum_{p\neq j}  \frac{\partial h}{\partial m}[j]
   \,\frac{\partial h}{\partial m}[p]  \prod_{n\neq j,p} h[n]\;\nabla m_{j,\beta}
   \otimes \nabla m_{p,\beta}\;d\xi
   \\
 &+ \int_{\mathbb{R}^{L+1}} \sum_{j=1}^J \,\frac{\partial h}{\partial m}[j]\,
 \prod_{n\neq j} h[n]\; \operatorname{Hess} m_{j,\beta}\,d\xi.
\end{align}

We will need the following auxiliary statement which we will deduce at the end of this section.
\begin{lemma}\label{p_bound_grad_and_hess_m}
It holds uniformly in $j=1,...,J$ and $\beta$ that
\begin{align}
\|\nabla m_{j,\beta}\| &\leq C,  \label{nablam}\\
\|\operatorname{Hess} m_{j,\beta}\| &\leq C.   \label{Hessm}
\end{align}
\end{lemma}
 From now on we write $\bar \xi_j := J^{-\frac{1}{2}} \xi \cdot \gamma^j$ for short. Denote ~$O:= \left\{|\xi|> J^{\frac{1}{2}}\delta\right\}$. The outer integral of \eqref{outerhess} becomes
\begin{align*}
& \left\|\operatorname{Hess}\, \int_{O}
\,\prod_{j=1}^J \,h(m_{j, \beta}, J^{-\frac{1}{2}} \xi \cdot \gamma^j)\, \,d\xi \;\right\|\nonumber\\
&\stackrel{\eqref{Hessg}-\eqref{Hessm}}{\lesssim}  \int_{O} \sum_{j=1}^J
      \left|\frac{\partial^2 h}{\partial m^2}[j]\right|\;
      \prod_{n\neq j} |h[n]|\;d\xi + \int_{O} \sum_{j=1}^J \,\left|\frac{\partial h}{\partial m}[j]\right|\,
      \prod_{n\neq j} |h[n]|\; d\xi  \\
    &\qquad + \int_{O} \sum_{j=1}^J \sum_{p\neq j} 
  \left|\frac{\partial h}{\partial m}[j]\right|\,
  \left|\frac{\partial h}{\partial m}[p]\right| \prod_{n\neq j,p} |h[n]|\;d\xi
  \\
&\stackrel{\eqref{globalboundh}-\eqref{secderh}}{\lesssim}
    \int_{O} \sum_{j=1}^J
   \prod_{n\neq j: \frac{n}{J}\in I_{k(\xi)}} \frac{1+ |\bar \xi_j|^2}{1+|\bar \xi_n|\,C_{\varepsilon}^{-1}}\;d\xi + \int_{O} \sum_{j=1}^J  \prod_{n\neq j: \frac{n}{J}\in I_{k(\xi)}} \frac{1+ |\bar \xi_j|}{1+|\bar \xi_n|\,C_{\varepsilon}^{-1}}\; d\xi \\
 &\qquad + \int_{O} \sum_{j=1}^J \sum_{p\neq j}  \prod_{n\neq j, p: \frac{n}{J}\in I_{k(\xi)}} \frac{(1+ |\bar \xi_j|)
 	(1+ |\bar \xi_p|)}{1+|\bar \xi_n|\,C_{\varepsilon}^{-1}}\;d\xi
 \\
&\stackrel{\eqref{gammadisc}}{\lesssim}
J A^{\frac{J}{L+2}-L-3}
\int_{O}  B(\xi)^{L+2}\;d\xi + \max_j |\gamma^j|^2 A^{\frac{J}{L+2}-L-5}
\int_{O}  |\xi|^2\, B(\xi)^{L+4}\;d\xi \\
& \qquad + J^2 A^{\frac{J}{L+2}-L-4}
\int_{O} B(\xi)^{L+2}\;d\xi +  J\,\max_j |\gamma^j|^2 A^{\frac{J}{L+2}-L-6}
\int_{O}  |\xi|^2\, B(\xi)^{L+4}\;d\xi \\
&\qquad +  J^{\frac{1}{2}}\,\max_j |\gamma^j| A^{\frac{J}{L+2}-L-4}
\int_{O}  |\xi|\, B(\xi)^{L+3}\;d\xi
\end{align*}
where 
\begin{equation}
A := \frac{1}{1+\varepsilon \,C_{\varepsilon}^{-1}}, \quad B(\xi) := \frac{1}{1+J^{-\frac{1}{2}} |\xi|\,c_{\gamma}\,C_{\varepsilon}^{-1}}.
\end{equation}
In the last step, we have collected like terms after application of the estimate \eqref{gammadisc}.
We also used Young's inequality once.
Observe that we always left exactly enough of the factors that were at our disposal, i.e.
$J/(L+2) - 1$ and $J/(L+2) - 2$ respectively, inside the integral to ensure integrability.
Performing a change of variables just as in the last step of the proof for \eqref{outerint},
we find that the right hand side goes to zero as $J\rightarrow \infty$ because exponential decay
beats polynomial growth. This proves \eqref{outerhess}.\\

For the inner integral of \eqref{innerhess}, we again use the representation via $h_2$ from \eqref{h2}.
In this case, we have the following
formulas for the derivatives with respect to $m$:
\begin{align*}
&\frac{\partial h}{\partial m} (m,z) \,=\, - z^2 \,\frac{\partial h_2}{\partial m}(m,z) \,h(m,z), \\
&\frac{\partial^2 h}{\partial m^2} (m,z) \,=\,
 \left(- z^2 \,\frac{\partial^2 h_2}{\partial m^2}(m,z) \,+\, z^4\,\left(\frac{\partial h_2}{\partial m}(m,z)\right)^2\right) \,h(m,z).
\end{align*}
Denote ~$I :=\left\{|\xi|\leq J^{\frac{1}{2}}\delta\right\}$. Using the bounds from \eqref{derhinside}, we find
\begin{align*}
\lefteqn{\left\|\,\operatorname{Hess}\,\int_{I} \,
 \prod_{j=1}^{J} \,h(m_{j,\beta}\,,\,J^{-\frac{1}{2}}\,\xi\cdot\gamma^{j}) \;d\xi\,\right\|}\\
&\stackrel{\eqref{Hessg}-\eqref{Hessm}}{\lesssim}
  \int_{I} \sum_{j=1}^J \bigg(\bar \xi_j^2\,
  \left|\frac{\partial^2 h_2}{\partial m^2}[j]\right| +\,\bar \xi_j^4\,\left|\frac{\partial h_2}{\partial m}[j]\right|^2\bigg) \prod_{n=1}^J \,|\exp(-\bar \xi_n^2 h_2[n])|\,d\xi \nonumber\\
 &+ \int_{I} \sum_{j=1}^J \,\sum_{p\neq j}\, \bar \xi_j^2
 \,\bar \xi_p^2  \left|\frac{\partial h_2}{\partial m}[j]\right|\,
 \left|\frac{\partial h_2}{\partial m}[p]\right|\prod_{n=1}^J\,|\exp(-\bar \xi_n^2 h_2[n])|\,d\xi\\
 &+  \int_{I} \sum_{j=1}^J \,\bar \xi_j^2\,
 \left|\frac{\partial h_2}{\partial m}[j]\right| \prod_{n=1}^J\,|\exp(-\bar \xi_n^2 h_2[n])|\,d\xi\\
&\stackrel{\eqref{derhinside},\eqref{h2zero}}{\lesssim}
  \max_j |\gamma^j|^2 \int_{I}
  \,|\xi|^2 \exp\left(-c J^{-1}\,\sum_{n:\frac{n}{J}\,\in\, I_{k(\xi)}} |\xi|^2 \,c_{\gamma}^2 \,\right) \,d\xi\\
&+(1+J^{-1})\,\max_j |\gamma^j|^4 \int_{I}
  \,|\xi|^4 \exp\left(-c J^{-1}\,\sum_{n:\frac{n}{J}\,\in\, I_{k(\xi)}} |\xi|^2 \,c_{\gamma}^2 \,\right) \,d\xi\\
&\lesssim \int_{\mathbb{R}^{L+1}}\, (|\xi|^2 + |\xi|^4) \exp\left(-c\,c_{\gamma}^2\,\frac{1}{L+2} |\xi|^2 \,\right) \,d\xi,
\end{align*}
which is finite. This completes the proof of Proposition~\ref{estimates on g} up to the verification of Lemma~\ref{p_bound_grad_and_hess_m}.

\begin{proof}[Proof of Lemma~\ref{p_bound_grad_and_hess_m}]
We recall that
\begin{equation*}
m_{j,\beta} = \int_{\mathbb{R}} z\, \exp(-\psi^*(\hat{m}_{j,\beta})+\hat{m}_{j,\beta}z-\psi(z)) \,dz.
\end{equation*}
Standard calculation yields
\begin{align*}
\nabla m_{j,\beta}   & \stackrel{\eqref{meanj}}{=} \operatorname{Var}(\mu_{m_{j,\beta}})\; \nabla \hat{m}_{j,\beta}, \\
\operatorname{Hess} m_{j,\beta}
 &= \operatorname{Var}(\mu_{m_{j,\beta}})\; \operatorname{Hess}\hat{m}_{j,\beta} \\ &\quad + \left(\int \,(z - m_{j,\beta})^3 \; \mu_{m_{j,\beta}}(dz)\right) \:
     \nabla \hat{m}_{j,\beta} \,\otimes\,\nabla \hat{m}_{j,\beta}.
\end{align*}
By the uniform estimates on $\operatorname{Var}(\mu_{m})$ and $\int \,(z-m)^3\, \mu_m(dz)$
in \eqref{boundvar} and \eqref{boundthirdm},
it remains to bound $\nabla \hat{m}_{j,\beta}$ and $\operatorname{Hess}\hat{m}_{j,\beta}$. Note that
\begin{align}
\hat{m}_{j,\beta} \,\stackrel{\eqref{dualmeanj}}{=}\,
\langle\hat{\beta}^{max}, \gamma^j\rangle \stackrel{\eqref{duality}}{=} \langle\nabla \bar{\psi}_J(\beta), \gamma^j\rangle = \partial_{\gamma^j} \bar{\psi}_J(\beta) \label{hatm}
\end{align}
where $\partial_\eta$ denotes the partial derivative in the direction of $\eta$. Thus for any $\eta \in \mathbb{R}^{L+1}$:
\begin{align}
\langle \nabla \hat{m}_{j,\beta}, \eta \rangle &= \partial_\eta \partial_{\gamma^j} \bar{\psi}_J(\beta) \leq \|\operatorname{Hess} \bar{\psi}_J\| |\eta||\gamma^j| \\
\langle \operatorname{Hess} \hat{m}_{j,\beta} \eta, \eta \rangle  &= \partial^2_\eta \partial_{\gamma^j} \bar{\psi}_J(\beta) \leq \|D^3 \bar{\psi}_J\| |\eta|^2|\gamma^j|
\end{align}
Since $|\gamma^j|$ is uniformly bounded (cf. \eqref{gammaj_approx}) and the Hessian and 3rd derivative of $\bar{\psi}_J$ are uniformly bounded, this concludes the proof of \eqref{nablam} and \eqref{Hessm}. 
\end{proof}

\section{Proof of Theorem~\ref{p_LSI_conditional}}\label{s_uniform_LSI}

The purpose of this section is to deduce Theorem~\ref{p_LSI_conditional}, which states that the conditional measure~$\mu(dx|Px=y)$ satisfies a uniform LSI. Let us outline how we proceed. In Section~\ref{s_basic_principles_LSI}, we state some basic principles of the LSI and introduce the two-scale criterion which is the principle that underlies our argument for deducing Theorem~\ref{p_LSI_conditional}. In Section~\ref{mainresultLSI}, we explain how those principles are applied to deduce the uniform LSI for the conditional measure~$\mu(dx|Px=y)$. In Section~\ref{s_gradient_parallel} we give the proof of an auxiliary result.

\subsection{Basic principles for the LSI}\label{s_basic_principles_LSI}

Four different principles underly our proofs of logarithmic Sobolev
inequalities in Section~\ref{mainresultLSI}. Three of these are standard results
that have proven to be useful for establishing LSI in many cases and that have been
known for a long time. The fourth principle is a more specialized criterion that has been successfully applied for deducing LSI for spin systems. It will guide our main strategy of proof while the other results
are needed to verify the assumptions of the criterion.
Let us forget for a moment the precise definitions of $X_N$ and $H$ and let us present the basic principles of the LSI in the setting of Euclidean spaces. So let~$X$ be Euclidean space or affine subspaces
of some Euclidean space. With~$\nabla$ and~$|\cdot|$ we denote the gradient and norm that is derived from the Euclidean structure of~$X$. We write
$\mathcal{P}(X)$ for the space of Borel probability measures on $X$.
\begin{definition}[LSI]\label{deflsi}
Let 
$\Phi(z) := z\,\ln z$. We say that $\nu \in \mathcal{P}(X)$ satisfies a logarithmic
Sobolev inequality $(\operatorname{LSI})$ with constant $\rho>0$ if for all smooth functions
$h:X\rightarrow \mathbb{R}_+$ it holds that
\begin{equation*}
\Ent( h \nu | \nu) := \int \Phi(h) \;\nu(dx) \,-\,\Phi\left(\int h\;\nu(dx)\right) \;\leq\,
 \frac {1}{\rho} \int \frac{1}{2h}\, |\nabla h|^2 \;\nu(dx).
\end{equation*}
In this case, we also use the notation~$ \operatorname{LSI}(\nu) \geq\rho$.
\end{definition}

The following tensorization principle has been known ever since the notion of LSI came up (see~\cite{Gro75}). It is the basic reason for why LSI is well-suited for high-dimensional systems.
\begin{lemma} [Tensorization principle] \label{prodpropLSI}
Given $\nu_n \,\in\, \mathcal{P}(X_N)$ for $n\,=\,1,...,N$. Then $\operatorname{LSI}(\nu_n)\geq \rho_n$ for all $n=1,...,N$ implies:
\begin{equation*}
 \operatorname{LSI}\left(\bigotimes_{n=1}^{N} \,\nu_n\right) \geq \min_n\, \rho_n.
\end{equation*}
\end{lemma}
We next recall two fundamental criteria for proving logarithmic Sobolev inequalities.
The first one is a simple perturbation result due to Holley and Stroock~\cite{HS}.
In the statement of this lemma, as well as later on in the text, we use $Z^{-1}$ to denote a constant
normalizing a given measure to unit mass.
\begin{lemma} [Holley-Stroock] \label{HStroock}
We assume that $\nu \,\in\,\mathcal{P}(X)$ satisfies $\operatorname{LSI}(\nu) \geq \rho$. For a bounded function
$\delta\psi:\,X\rightarrow\mathbb{R}$, define
a measure $\widetilde{\nu} \,\in\,\mathcal{P}(X)$ that is absolutely continuous with respect to $\nu$ via
\begin{equation*}
\frac{d\widetilde{\nu}}{d\nu}(x) \,=\, Z^{-1}\,\exp[-\delta\psi(x)].
\end{equation*}
Then $\operatorname{LSI}(\widetilde{\nu}) \geq \rho\,\exp\,[-2\operatorname{osc}(\delta\psi)]$.
Here $\operatorname{osc}(\delta\psi) = \sup_X \delta\psi \,-\,\inf_X \delta\psi$ stands for the total oscillation of the
perturbation.
\end{lemma}
The second criterion is due to Bakry and \'Emery~\cite{BE}. It says that a uniformly strictly convex Hamiltonian implies $\operatorname{LSI}$.
\begin{lemma} [Bakry-\'Emery]\label{BakryEmery}
Let $\nu \,\in\,\mathcal{P}(X)$ be absolutely continuous with respect to the
Hausdorff measure $\mathcal{H}$ on $X$.
If the Hamiltonian $H$ of the measure $\nu$, given by
\begin{equation*}
H(x) :=\, -\ln\,\frac{d\nu}{d\mathcal{H}}(x),
\end{equation*}
is twice continuously differentiable and uniformly convex with lower bound $\lambda$, i.e.
\begin{equation*}
\forall\, x\in X \quad \forall\, v \in T_x X \quad \langle v,\operatorname{Hess} H(x)\,v\rangle_{T_x X} \,\geq\, \lambda \,|v|_{T_x X}^2,
\end{equation*}
then $\operatorname{LSI}(\nu)\,\geq\,\lambda$.
\end{lemma}
Proofs of the facts mentioned so far can be found for example in~\cite{GZ} or in the nice introduction to both
spectral gap and logarithmic Sobolev inequalities~\cite{L}.
As pointed out above, we will in addition need
the two-scale criterion that was presented in~\cite{OR}
and which is also contained, in a slightly different formulation, in~\cite{GORV}.
We first define a decomposition of measures analogous
to Definition~\ref{d_disintegration_canonical_ensemble} in the setting of a product space.
\begin{definition} \label{defdecomp}
Let $\nu\in\mathcal{P}(X_1\times X_2)$ be a measure with smooth positive probability density function
with respect to Hausdorff measure. We decompose $\nu$ into a family of conditional measures
$\{\nu(dx_1|x_2)\}_{x_2\in X_2} \,\subset\, \mathcal{P}(X_1)$ and the corresponding marginal measure $\bar{\nu} \,\in\, \mathcal{P}(X_2)$.
This decomposition is such that for all measurable $h: X_1 \times X_2 \rightarrow \mathbb{R}$:
\begin{equation*}
 \int_{X_1\times X_2} h\,d\nu \,=\,\int_{X_2} \int_{X_1} h(x_1,x_2) \,\nu(dx_1|x_2) \,\bar{\nu}(dx_2).
\end{equation*}
\end{definition}
The two-scale criterion reads as follows.
\begin{lemma} [Two-scale criterion for LSI]\label{OR}
Let $\nu\in\mathcal{P}(X_1\times X_2)$ be a measure with twice continuously differentiable Hamiltonian $H$.
Assume that there exist constants $\rho_1,\rho_2>0$ such that
\begin{equation*}
\operatorname{LSI}(\nu(dx_1|x_2)) \geq\rho_1 \quad \text{uniformly in } x_2\in X_2,
\end{equation*}
\begin{equation*}
\operatorname{LSI}(\bar{\nu}) \geq \rho_2.
\end{equation*}
Assume furthermore that
\begin{equation}
\frac{1}{\rho_1}\,\frac{1}{\rho_2}\, \sup_{X_1\times X_2}\,|\nabla_{X_1} \nabla_{X_2} H(x)|^2 \,=\,
\kappa <\infty. \label{coupling}
\end{equation}
Here,
\begin{equation*}
|\nabla_{X_1} \nabla_{X_2} H(x)| \,=\, \sup\,\{\langle\operatorname{Hess} H(x)\, u,v \rangle|\,
  u\in T_xX_1, v\in T_xX_2,|u|=|v|=1\},
\end{equation*}
which is finite if $\operatorname{Hess} H$ is bounded.
Then
\begin{equation*}
\operatorname{LSI}(\nu) \,\geq\, \frac{1}{2} \left(\rho_1 + (1+\kappa)\rho_2 \,-\,\sqrt{(\rho_1 +(1+\kappa) \rho_2)^2
- 4\rho_1\rho_2}\right).
\end{equation*}
\end{lemma}
Lemma \ref{OR} says that $\operatorname{LSI}$ for conditional measures and corresponding marginal may - under the coupling assumption
\eqref{coupling} - be combined to yield a $\operatorname{LSI}$ for the full measure.
A proof of the two-scale criterion can be found in~\cite{OR} where it is stated as Theorem $2$.

\subsection{Uniform LSI for conditional measures} \label{mainresultLSI}
In this section we explain how the basic principles of Section~\ref{s_basic_principles_LSI} are used to deduce Theorem~\ref{p_LSI_conditional}. \\

We start with showing that the LSI for the measures~$\mu(dx |Px=y)$ formulated in~\ref{p_LSI_conditional} coincides with the Definition~\ref{deflsi} of Section~\ref{s_basic_principles_LSI}. For this purpose, let us discuss the affine spaces and the Euclidean structures that are involved. The measure~$\mu(dx |Px=y)$ lives on the affine spaces~$$P^{-1}(\left\{ y \right\})= \left\{x \in X_N \ | \ Px=y \right\} \subset X_N = \mathbb{R}^N \qquad \mbox{for } y \in Y_M.$$
Let us now have a closer look at the gradient~$\nabla_{\parallel}$ on the space~$P^{-1}(\left\{ y \right\})$.
\begin{lemma}\label{p_fluctuation_gradient}
Let~$f: X_N \to \mathbb{R}$ be a smooth function. Let~$\nabla f$ be the gradient inherited from the standard Euclidean structure on~$X_N$. Then the gradient~$\nabla_{\parallel}$ on~$\ker P$ and the gradient ~$\nabla_{\perp}$ on ~$(\ker P)^{\perp} = \Image NP^t$ are given by
\begin{align}
\label{e_fluctuation_gradient}
\nabla_{\parallel} f = \nabla f - \nabla_{\perp} f \quad  \mbox{and} \quad \nabla_{\perp} f = NP^t (PNP^t)^{-1} P \nabla f.
\end{align}
\end{lemma} 
We postpone the proof of Lemma~\ref{p_fluctuation_gradient} to Section~\ref{s_gradient_parallel}. Using Lemma~\ref{p_fluctuation_gradient} one sees that the LSI~\eqref{e_d_LSI} of Theorem~\ref{p_LSI_conditional} coincides with the Definition~\ref{deflsi}.\\

Now, let us explain the main idea of the proof of Theorem~\ref{p_LSI_conditional}. As a first attempt to deduce the uniform LSI for the conditional measure~$\mu(dx| Px=y)$ one could try to apply the method of~\cite{GORV}. For a better illustration, let us now outline the basic idea of the approach of~\cite{GORV}: The lattice~$\left\{1, \ldots, N \right\} $ is decomposed into~$M$ many blocks~$B(l)$ each containing~$K$-many sites. The coarse-graining operator~$P: \mathbb{R}^N \to \mathbb{R}^M$ in~\cite{GORV} projects spin configurations $x$ onto the local means over the blocks. The uniform LSI for the measure~$\mu(dx|Px=y)$ is then deduced in the following way: Since the coarse-graining is completely local and the
Hamiltonian is uncoupled, $\mu(dx|Px=y)$ is a
product measure of conditional measures that condition on the mean over a block i.e.
\begin{align}
  \mu(dx|Px=y) = \otimes_{l=1}^M \mu \left( d(x_i)_{i \in B(l)}| \frac{1}{K} \sum_{i \in B(l)} x_i = y_l \right).
\end{align}
Hence by the tensorization principle (cf.~Lemma~\ref{prodpropLSI}), it is sufficient to show a uniform LSI for these measures. For convenience, let us only consider at the first block $B(1)=\left\{1, \ldots , K \right\}$. This block is further subdivided into~$R$ many blocks~$\tilde B(\tilde l)$ each containing~$J$ sites. Again, in~\cite{GORV} one considers the coarse-graining operator~$\tilde P: \mathbb{R}^K \to \mathbb{R}^R$ that maps a configuration~$(x_i)_{i \in B(1)} \in \mathbb{R}^K$ on the block $B (1)= \left\{1, \ldots , K \right\}$ onto the mean values 
\begin{align}
  \tilde y_{\tilde l} = \frac{1}{J} \sum_{i \in \tilde B(\tilde l)} x_i \qquad \tilde l \in \left\{ 1, \ldots , R \right\}.  
\end{align}
over the blocks of size~$J$. As this operator $\tilde P$ is orthogonal, disintegration of measures (see e.g.~Definition~\ref{d_disintegration_canonical_ensemble}) yields that the measure~$\mu \left( dx_1, \ldots, dx_K | \frac{1}{K} \sum_{i =1}^K x_i = y_1 \right)$ can be decomposed into conditional measures and a marginal i.e.
\begin{align}
  \mu & \left( dx_1, \ldots, dx_K | \frac{1}{K} \sum_{i =1}^K x_i = y_1 \right) \\
& = \mu \left( dx_1, \ldots, dx_K | \tilde Px = \tilde y \right) \bar \mu(d \tilde y_{1}, \ldots , \tilde y_{R})\\ 
& = \left( \otimes_{l=1}^{R} \mu \left( d(x_i)_{i \in \tilde B(l)} | \frac{1}{J} \sum_{i \in \tilde B(\tilde l)} x_i =\tilde y_{\tilde l}   \right)  \right)  \bar \mu \left( d \tilde y_{1}, \ldots , d \tilde y_{R} \right).
\end{align}
The uniform LSI for the measures~$$ \mu  \left( dx_1, \ldots, dx_K | \frac{1}{K} \sum_{i =1}^K x_i = y_1 \right) $$  now follows from an application of the two-scale criterion (see Lemma~\ref{OR}). Indeed, a combination of the tensorization principle and of Holley-Stroock (cf.~Lemma~\ref{prodpropLSI} and Lemma~\ref{HStroock}) yields that the conditional measure 
$$
\mu \left( dx_1, \ldots, dx_K | \tilde Px = \tilde y \right)
$$
satisfies a $\LSI(\varrho_1)$ with constant~$\varrho_1>0$ that only depends on~$J$. Additionally, the marginal measure~$\bar \mu (d \tilde y)$ satisfies a uniform LSI if~$J$ is large enough, which follows from the fact that the Hamiltonian~$\bar H(\tilde y)$ of the marginal $\bar \mu (d \tilde y) = \frac{1}{Z} \exp (- \bar H (\tilde y)) d \tilde y$ is uniformly convex for large enough~$J$. Hence, Lemma~\ref{OR} yields that the measures
\begin{align*}
  \mu \left( d(x_i)_{i \in B(l)},| \frac{1}{K} \sum_{i \in B(l)} x_i = y_l \right)
\end{align*}
satisfy a LSI with constant~$\varrho>0$ that only depends on~$J$. By choosing~$J$ large enough but fixed, this yields that the LSI constant~$\varrho>0$ uniformly in~$K$. Hence finally, by choosing~$K$ large enough but fixed, this yields that the measures~$\mu(dx|Px=y)$ satisfy a LSI with constant~$\varrho>0$ that is uniform in~$N$ and in~$y$.\\

We want to note that this approach needs the technical assumption that $N$ is an integer multiple of $M$. However, it will be apparent that the same arguments apply whenever the $N$ spins can be partitioned into $M$ blocks in such a way that
each block contains a number of spins greater than some critical value $K^*$.\\

Unfortunately, if one wants to apply the same argument to the coarse-graining operator~$P : X_N \to Y_M$, that projects onto splines, one runs into problems. It is crucial for the method of~\cite{GORV} that the coarse-graining operator~$P$ is local in the sense that spins from one block (let's say the interval $\left(0,\frac{1}{M}\right)$)
are independent of the conditioning on any other block. This is important for both the factorization and
the convexification part of the argument. In this respect, the spline spaces for $L\geq 1$ are not ideally suited for an application
of the two-scale criterion. Even though splines do possess a localized
basis of B-splines, these functions overlap.\\

We will circumvent this problem in the following way. Similar to the proceeding in Section~\ref{s_convexification}, we consider the space~$Y_M^{DG}$ of piecewise polynomials i.e.~(cf.~Definition~\ref{d_Y_pol_and_H})
\begin{align} \label{d_Y_pol_2}
 Y^{DG}_{M} & := \left\{ y\in L^2(\mathbb{T}) \ | \ \,\forall m=1,\ldots,M:  \right.\\
& \qquad \qquad \left. y|_{\left(\frac{m-1}{M},\frac{m}{M}\right)}
                         \,\text{polynomial of degree $\leq$ L}   \right\} . 
\end{align}
We then consider the orthogonal projection $Q_M: \mathbb{R}^N \rightarrow Y_M^{DG}$  onto $Y_M^{DG}$ in $L^2(\mathbb{T})$. The advantage of~$Q_M$ over~$P$ is that the coarse-graining operator~$Q_M$ is purely local in the sense we described above. This allows us to use the same strategy as in~\cite{GORV} to deduce that the conditional measures~$\mu(dx| Q_Mx =y)$ satisfy a LSI that is uniform in~$y$ and in~$N$ (see Theorem~\ref{LSIDG} and Lemma~\ref{p_product_structure_conditional_Q} below). In the second step of the argument, we deduce the uniform LSI of the conditional measures~$\mu(dx| Px=y)$ from the uniform LSI of the measures~$\mu(dx| Q_Mx =y)$ via another application of the two-scale criterion for the LSI (see Lemma~\ref{OR}).\\

The next statement provides the main ingredient of the proof of Theorem~\ref{p_LSI_conditional}.
\begin{theorem}[Uniform LSI for~$\mu (dx|Q_Mx=y)$]\label{LSIDG}
Let $Y^{DG}_{M}$ be the space of piecewise polynomials (see~\eqref{d_Y_DG} or Definition~\ref{d_Y_pol_and_H}). Let $Q_M: \mathbb{R}^N \rightarrow Y_M^{DG}$ denote the orthogonal projection onto $Y_M^{DG}$ in $L^2(\mathbb{T})$. Then the conditional measures~$\mu(dx | Q_Mx=y)$ satisfy~$LSI(\varrho_Q)$ with a constant~$\varrho_Q>0$ uniform in~$N,M$ and~$y \in Y_M^{DG}$.  
\end{theorem}
\begin{remark}
Both in Theorem~\ref{p_LSI_conditional} and in Theorem \ref{LSIDG}, the case $L=0$ corresponds to the result in~\cite{GORV}.
In Theorem \ref{LSIDG}, the space of observables is of course of dimension $(L+1)M$, not $M$.
\end{remark}

We will now state the proof of Theorem~\ref{LSIDG}. Because the coarse-graining operator~$Q_M$ is local and the Hamiltonian~$H$ has no interaction between different sites a straight forward calculation shows that the conditional measure~$\mu(dx|Q_Mx=y)$ has the following product structure. In the following lemma, we use notation introduced in Section~\ref{s_strict_convexity_bar_H_DG}.

\begin{lemma}\label{p_product_structure_conditional_Q}
We decompose the lattice~$\left\{1, \ldots, N \right\}$ into~$M$ many blocks $$B(m) = \left\{ (m-1)K+1, \ldots, mK \right\}, \qquad 1 \leq m \leq M,$$ each consisting of~$K$ many sites. For convenience, we denote with~$x^{(m)}$ the vector~$x^{(m)}= (x_i)_{i \in B(m)}$. For clarity, we denote $\mu^N$ instead of $\mu$ the Gibbs measure on $\mathbb{R}^N$. Then it holds that 
\begin{equation}
 \mu^N( dx |Q_Mx= \alpha) = \bigotimes_{m=1}^{M}  \mu^K ( dx^{(m)} | Q_1x^{(m)}= \alpha^{(m)})
\end{equation}
\end{lemma}
We skip the the simple proof of Lemma~\ref{p_product_structure_conditional_Q} and continue with deducing Theorem~\ref{LSIDG}.

\begin{proof}[Proof of Theorem~\ref{LSIDG}]
By the tensorization principle (cf.~Lemma~\ref{prodpropLSI}) and Lemma~\ref{p_product_structure_conditional_Q} it suffices to show that the conditional measures 
\begin{align} 
 \mu^K ( dx^{(m)} | Q_1x^{(m)}= \alpha^{(m)})
\end{align}
satisfy a uniform LSI. The strategy is to apply the two-scale criterion for LSI (see Lemma~\ref{OR}). We have to carry out several steps in order to verify the hypotheses of Lemma~\ref{OR}. As usual, the polynomial degree $L$ is fixed for the entire argument and we will not always mention the dependencies of constants on $L$ explicitly.\\
 
Without loss of generality we may only consider the first block~$B(1)$. For convenience we will write~$x \in \mathbb{R}^K$ instead of~$x^{(1)}$ and $\alpha$ instead of~$\alpha^{(1)}$. We also introduce the short notation
\begin{align} \label{e_d_mu_alpha_K}
\mu_\alpha^K (d x ) := \mu^K ( dx | Q_1x= \alpha).
\end{align}
From now on we assume that~$K$ is of the form $K=RJ$,
where $J\in \mathbb{N}$ is large but fixed and $R\in\mathbb{N}$ is arbitrary. The measure
$\mu_\alpha^K (d x )$ is defined on the space (cf.~\eqref{e_d_R_M_alpha})
\begin{align}
  \mathbb{R}^K_{1,\alpha} := \{x\in \mathbb{R}^K|\,Q_{1}x =\alpha\} .
\end{align}
In order to apply Lemma~\ref{OR}, we have to decompose the space~$\mathbb{R}^K_{1,\alpha}$ into an orthogonal sum of two spaces. One space will describe the mesoscopic profile of~$x$ and the other one will describe the fluctuations around this profile. For this purpose we use the coarse-graining operator $Q_{R}$. Compared to~\cite{GORV} there is a small but unproblematic technical complication: For $L\geq 1$, $Q_{R}KQ_{R}^t$ is not equal to the identity on $Y_{R}^{DG}$. However, we have already seen in Lemma~\ref{PNP} that~$Q_{R}KQ_{R}^t$ is close to the identity for large~$J$. This means that dealing with the resulting approximation error is unproblematic. Yet, we have to be a little bit careful when decomposing the space~$ \mathbb{R}^K_{1,\alpha} $.\\

If $J$ is sufficiently large, then by~\eqref{approxid3} there is for any given $x \in \mathbb{R}^J$ a unique $y \in Y_{1}^{DG}$ such that
\begin{equation}\label{e_solution_PNP}
Q_{1}\,(x - JQ_{1}^ty) \,=\, 0.
\end{equation}
The equation~\eqref{e_solution_PNP} yields the desired orthogonal decomposition of $x$ into a fluctuation part from
$Ker \,Q_{1}$ and a macroscopic part:
\begin{equation*}
 x\,=\, (x-JQ_{1}^ty)\,+\,JQ_{1}^ty.
\end{equation*}
The respective formula for more than one block follows immediately from~\eqref{P3}. Namely, we write~$ \mathbb{R}^K \ni x = (x^{(1)}, \ldots, x^{(R)})$, where~$x^{(r)} \in \mathbb{R}^J$. Then we define for~$1 \leq r \leq R$
\begin{equation*}
 x^{(r)}= \underbrace{(x^{(r)}-JQ_{1}^ty^{(r)})}_{=: x_\parallel^{(r)}}+ \underbrace{JQ_{1}^ty^{(r)}}_{=: x_{\perp}^{(r)}}.
\end{equation*}
 Then we get that any~$x \in \mathbb{R}^K$ can be uniquely written as
 \begin{align}
   \label{e_orthogonal_decomposition_Q}
   x= x_{\parallel} + x_{\perp}, 
 \end{align}
where~$x_{\parallel} \in Ker \,Q_{R}$ and~$x_{\perp} \in Im \,KQ_{R}^t$. Hence we obtain the orthogonal decomposition
\begin{equation}
\mathbb{R}^K \,=\, V \oplus W
\end{equation}
where $V:= Ker\, Q_{R}$ and $W:= Im\, KQ_{R}^t$.\\

The key observation here is that $Y_1^{DG}$ is a linear subspace of $Y_R^{DG}$, which gives the inclusion $V \subset Ker \,Q_{1}$. This allows us to decompose the affine subspace $\mathbb{R}^K_{1,\alpha} \subset \mathbb{R}^K$ as
\begin{equation}
 \mathbb{R}^K_{1,\alpha}\,=\, V \oplus W_{\alpha}, \label{facfibre}
\end{equation}
where
\begin{equation*}
W_{\alpha} := KQ_{R}^t\,\{y\in Y_{R}^{DG}| Q_{1}KQ_{R}^ty =\alpha\}.
\end{equation*}
Now, following Definition~\ref{defdecomp} we  decompose $\mu^K_{\alpha}(dx)$ (cf.~\eqref{e_d_R_M_alpha}) with respect to the
factorization \eqref{facfibre} into the family of conditional measures
$\left\{ \mu^K(d x_{\parallel} | x_{\perp}) \right\}_{x_{\perp}} \subset\, \mathcal{P}(V)$ and the marginal
$\bar{\mu}^K_{\alpha} (dx_{\perp}) \,\in \mathcal{P}(W_{\alpha})$. \\

We will apply the two-scale criterion to prove Theorem \ref{LSIDG}. Hence, we need to show that there are $\rho_1,\rho_2 > 0$, independent of $R$ (and hence $K$) and $\alpha$, such that
\begin{equation}
\forall x_{\perp}\,\in\,W_{\alpha}: \: \operatorname{LSI}(\mu^K( d x_{\parallel}| x_{\perp})) \geq \rho_1, \label{lsicond}
\end{equation}
and
\begin{equation}
\operatorname{LSI}(\bar{\mu}^K_{\alpha} (d x_{\perp})) \geq \rho_2 \label{lsimarg}.
\end{equation}
Additionally, we have to show that
\begin{equation}
\frac{1}{\rho_1}\,\frac{1}{\rho_2}\, \sup_{\mathbb{R}^K_{1,\alpha}}\,|\nabla_V \nabla_{W_{\alpha}} H_K|^2 \,\leq\, \kappa \,<\,\infty, \label{coupling2}
\end{equation}
for some constant $\kappa$ that is independent of $K$ and $\alpha$.\\

Let us now deduce~\eqref{lsicond}. We apply Lemma~\ref{p_product_structure_conditional_Q} to our situation by setting~$N=K$ and~$K=J$. This yields that the measure
$\mu^K (d x_{\parallel} | x_{\perp})$ is a product measure i.e.
\begin{equation}
\mu^K (d x_{\parallel} | x_{\perp})  = \bigotimes_{r=1}^R \mu^J (d x_{\parallel}^{(r)} | x_{\perp}^{(r)}). \label{prodcond2}
\end{equation}
Since $J$ is a fixed finite integer, the Hamiltonian of $\mu^{J}$ is just a bounded perturbation of a function that is
uniformly convex with lower bound one. Consequently, we get from a combination of Lemma~\ref{HStroock} and Lemma~\ref{BakryEmery} that
\begin{equation*}
\operatorname{LSI} \left( \mu^{J} (d x_{\parallel}^{(r)} | x_{\perp}^{(r)} )\right) \geq \exp\,(-2J \operatorname{osc}(\delta\psi)) =: \rho_1 > 0.
\end{equation*}
Now, a combination of~\eqref{prodcond2} and Lemma~\ref{prodpropLSI} yields~$ \eqref{lsicond}$.\\

Let us now turn to~\eqref{lsimarg}. The strategy is to show that the Hamiltonian of the marginal measure~$\bar{\mu}^K_{\alpha} (d x_{\perp})$ is uniformly strictly convex. The desired statement~\eqref{lsimarg} follows then from the Bakry-\'Emery criterion (cf.~Lemma~\ref{BakryEmery}). We start with observing the the Hamiltonian~$\hat H_{W_\alpha}$ of the measure~$\bar{\mu}^K_{\alpha} (d x_{\perp})$ is given by
\begin{align}
  \hat H_{W_\alpha} (x_{\perp}) & := - \ln \frac{ d \bar{\mu}^K_{\alpha}}{dx_{\perp}} (x_{\perp}) \\
  & =- \ln \frac{1}{Z}\int_V \exp(-H_K(x_{\parallel}+ x_{\perp})) dx_{\parallel}.
\end{align}
Here as usual, $dx_{\parallel}$ is the Hausdorff measure of appropriate dimension on $V$. Using definition~\eqref{e_def_coarse_grained_Hamiltonian_DG} of~$\bar H_{Y_R^{DG}}$, the last identity yields
\begin{align}\label{e_connecting_hat_H_with_H_DG}
  \hat H_{W_\alpha} (x_{\perp}) = K \bar H_{Y_R^{DG}} (Q_{R} x_{\perp}) + \ln Z.
\end{align}
where $Z$ is a constant that accounts for different normalization constants. From Theorem~\ref{p_strict_convexity_bar_H_DG}, we know that~$\bar H_{Y_R^{DG}}$ is uniformly strictly convex, provided~$J$ is large enough. Now, using~\eqref{e_connecting_hat_H_with_H_DG} we will transfer the convexity from~$\bar H_{Y_R^{DG}}$ to~$\hat H_{W_\alpha}$. Applying Theorem~\ref{p_strict_convexity_bar_H_DG}, we get using the chain rule that for $J\geq J^*$ and arbitrary  $u\in T_{x_{\perp}} W_\alpha \subset Im KQ_R^t$, 
\begin{align}
\langle u,\operatorname{Hess}  \hat {H}_{W_\alpha}(x_{\perp}) \,u\rangle_{T_{x_{\perp}} W_\alpha}
&=  K \langle Q_R u , \operatorname{Hess}  \bar H_{Y_R^{DG}} (Q_Rx_{\perp}) \, Q_R u\rangle_{L^2}\\
&\geq 2 K \lambda |Q_R u|_{L^2}^2\\
 &\overset{\eqref{e_QtKQ}}{=} 2K \lambda \left(1+ O\left(\frac{1}{J^2}\right)\right) |u|_{L^2}^2, \\
 &= 2\lambda\left(1+ O\left(\frac{1}{J^2}\right)\right)\, |u|_{T_{x_{\perp}} W_\alpha}^2, 
\end{align}
where we used estimate \eqref{e_QtKQ} from below. This yields the uniform strict convexity of~$\hat H_{W_\alpha}$. Thus, the Bakry-\'Emery criterion (cf.~Lemma~\ref{BakryEmery}) implies~\eqref{lsimarg} with constant $\rho_2\,:=\,\lambda$.\\

Finally, let us consider~\eqref{coupling2}. We set $\kappa:= \frac{1}{\rho_1}\frac{1}{\lambda} \|\operatorname{Hess} H_K\|^2$. It is immediate from
the explicit form of $H_K$ (the $C^2(\mathbb{R})$-bound for $\delta\psi$ to be precise) that~$\kappa$ is bounded independently of $K$.\\

Overall, we may hence apply Lemma \ref{OR} which yields that for $J\geq J^*$:
\begin{equation*}
\operatorname{LSI}(\mu^K_{\alpha}) \geq \frac{1}{2}
 \left(\rho_1 \,+\,(1+\kappa) \lambda \,-\,\sqrt{(\rho_1 \,+\,(1+\kappa) \lambda)^2 \,-\, 4\rho_1\lambda} \right),
\end{equation*}
which is bounded from below uniformly in~$K$.

\end{proof}

Let us now turn to the proof of Theorem~\ref{p_LSI_conditional}. With Theorem~\ref{LSIDG} at hand, the proof of Theorem~\ref{p_LSI_conditional} consists of an application of the two-scale criterion (see Lemma~\ref{OR}). The argument is very similar to the proof of Theorem~\ref{LSIDG}.  
\begin{proof}[Proof of Theorem~\ref{p_LSI_conditional}]
We recall
the orthogonal decomposition induced by $Q_M$ from the proof of Theorem~\ref{LSIDG}: if $K$ is sufficiently large, any~$x \in \mathbb{R}^N$ can be uniquely written as
\begin{align}
x= x_{\parallel} + x_{\perp}, 
\end{align}
where~$x_{\parallel} \in Ker \,Q_{M}$ and~$x_{\perp} \in Im \,NQ_{M}^t$. In short,
\begin{equation}
\mathbb{R}^N \,=\, V \oplus  W
\end{equation}
where $V:= Ker\, Q_{M}$ and $W:= Im\, NQ_{M}^t$. \\

The key observation here is that $Y_{M}$ is a linear subspace of $Y^{DG}_{M}$, which gives the inclusion $V \subset Ker \,P$. This allows us to decompose the affine subspace $P^{-1}(\{y\}) \subset \mathbb{R}^N$ as
\begin{equation}
 \{x\in \mathbb{R}^N| Px =y\} \,=\, V
  \oplus W_y. \label{factsplinefiber}
\end{equation}
where
\begin{equation}W_y:= NQ_{M}^t \{\beta \in Y_{M}^{DG}\,| P NQ_{M}^t\beta =y\}.
\end{equation} 
Now, following Definition~\ref{defdecomp} we  decompose $\mu^N(dx|Px=y)$ with respect to the
factorization \eqref{factsplinefiber} into the family of conditional measures
$\left\{ \mu^N(d x_{\parallel} | x_{\perp}) \right\}_{x_{\perp}} \subset\, \mathcal{P}(V)$ and the marginal
$\bar{\mu}^N (dx_{\perp} | Px_{\perp} = y) \,\in \mathcal{P}(W_{y})$. \\

Let us compare this decomposition to the decomposition used in the proof of Theorem~\ref{LSIDG}. There, the constraint determining the affine subspace stemmed
from prescribing a global polynomial. Here,  it comes from prescribing a spline on the same mesh.\\

We will apply the two-scale criterion to prove Theorem \ref{p_LSI_conditional}. Hence, we need to show that there are $\rho_{DG}, \bar \rho > 0$, independent of $M$ (and hence $N$) and $y$, such that
\begin{equation}
\forall x_{\perp}\,\in\,W_{y}: \: \operatorname{LSI}(\mu^N( d x_{\parallel}| x_{\perp})) \geq \rho_{DG}, \label{lsicond_2}
\end{equation}
and
\begin{equation}
\operatorname{LSI}(\bar{\mu}^N (d x_{\perp}|Px_{\perp}=y)) \geq \bar \rho \label{lsimarg_2}.
\end{equation}
Additionally, we have to show that
\begin{equation}
\frac{1}{\rho_{DG}}\,\frac{1}{\bar \rho}\, \sup_{P^{-1}(\{y\})}\,|\nabla_V \nabla_{W_{y}} H_N|^2 \,\leq\, \kappa \,<\,\infty, \label{coupling2_2}
\end{equation}
for some constant $\kappa$ that is independent of $N$ and $y$.\\

Let us now turn to~\eqref{lsicond_2}. Note that
\begin{align}
 \mu^N(dx_{\parallel}| x_{\perp}) = \mu^N(dx| Q_M x = Q_M x_{\perp}).  
\end{align}
Using the last representation, it follows from Theorem~\ref{LSIDG} that the conditional measure satisfies a $\LSI(\varrho_{DG})$ with~$\varrho_{DG}>0$ uniformly in~$N$ and~$x_{\perp}$.\\

Let us now turn to~\eqref{lsimarg_2}. The proof follows the same strategy of the proof of \eqref{lsimarg} in the proof of Theorem \ref{LSIDG}: compute the Hamiltonian~$\hat H_{W_y}$ of the measure~$\bar \mu^N(d x_{\perp} | P x_{\perp}=y)$, relate it to $H_{Y_M^{DG}}$, transfer the uniform convexity of the latter to the former, and apply Bakry-\'Emery criterion (cf.~Lemma~\ref{BakryEmery}) to get the desired statement with $\bar \rho := \lambda$. We leave the details as an exercise. \\

Finally, let us consider~\eqref{coupling2_2}. We set $\kappa:= \frac{1}{\rho_{DG}}\frac{1}{\lambda} \|\operatorname{Hess} H_N\|^2$. It is immediate from
the explicit form of $H_N$ (the $C^2(\mathbb{R})$-bound for $\delta\psi$ to be precise) that~$\kappa$ is bounded independently of $N$.\\

Overall, we may hence apply Lemma \ref{OR} which yields that for $K\geq K^*$:
\begin{equation*}
  \operatorname{LSI}(\mu^N(dx|Px =y))  \geq \rho,
\end{equation*}
where the constant
\begin{equation*}
\rho:= \frac{1}{2} \left(\rho_{DG} \,+\,(1+\kappa) \lambda \,-\,\sqrt{(\rho_{DG} \,+\,(1+\kappa) \lambda)^2 \,-\, 4\rho_{DG}\lambda} \right).
\end{equation*}
is uniformly bounded from below in $N$.
\end{proof}

In the proof of Theorem~\ref{LSIDG} and Theorem \ref{p_LSI_conditional} above, we used the estimate \eqref{e_QtKQ} from below. Let us state and prove this estimate now.

\begin{lemma}\label{p_QtKQ} If $u \in KQ_R^t Y_R^{DG}$, then
	\begin{equation}\label{e_QtKQ}
	|Q_R u|_{L^2}^2 = \left(1+ O\left(\frac{1}{K^2}\right)\right)\,|u|_{L^2}^2
	\end{equation}
\end{lemma}

\begin{proof} By assumption, $u = KQ_R^t \beta$ for some $\beta \in Y_R^{DG}$, so
	\begin{align}KQ_R^t (Q_R KQ_R^t)^{-1} Q_R u = KQ_R^t \beta = u. 
	\end{align}
	Using this and \eqref{approxid3}, we get
	\begin{eqnarray*}
	\lefteqn{\,\langle Q_R u , Q_R u\rangle_{L^2}}  \\
	&=&  \,\langle (Q_R K Q_R^t)^{-1} Q_R u , Q_R u\rangle_{L^2} +  \,\langle (\id_{Y_R^{DG}} - (Q_R K Q_R^t)^{-1}) Q_R u , Q_R u\rangle_{L^2} \\
	&\geq& \,\langle KQ_R^t (Q_R K Q_R^t)^{-1} Q_R u , u\rangle_{L^2}  -  \|\id_{Y_R^{DG}} - (Q_R K Q_R^t)^{-1}\| |Q_R u|_{L^2}^2 \\
	&\geq&  \,\langle u , u\rangle_{L^2} -  O\left(\frac{1}{K^2}\right) |u|_{L^2}^2  = \left(1+ O\left(\frac{1}{K^2}\right)\right)\,|u|_{L^2}^2
	\end{eqnarray*}
This reverse inequality is trivial.
\end{proof}

\subsection{Proof of Lemma~\ref{p_fluctuation_gradient}: Determining the gradient on~$\ker P$}\label{s_gradient_parallel}

The proof of Lemma~\ref{p_fluctuation_gradient} needs the following auxiliary result.
\begin{lemma}\label{p_invertible_PNPt}
It holds that
	\begin{align}\label{e_p_operator_estimate}
	\| PNP^t  - \id_{Y_M} \| \lesssim \frac{1}{K^2}.
	\end{align}
In particular, if~$K$ is large enough then $P NP^t: Y_M \rightarrow Y_M$ is invertible. 
\end{lemma}
\begin{proof}[Proof of Lemma~\ref{p_invertible_PNPt}]
Recall that we split up our interval~$[0,1]$ into~$M$ subintervals of length~$1/M$ (and~$N=MK$). It then follows from~\eqref{approxid3} that for~$K$ large enough the operator~$Q_MNQ_M^t$ is close to the identity~$\id_{Y_M^{DG}}$ in the sense that 
\begin{align}\label{e_dizdar_operator_estimate}
  \| Q_MNQ_M^t  - \id_{Y_M^{DG}} \| \leq \frac{C}{K^2}.
\end{align}
Now, since $Y_M \subset Y_M^{DG}$, $P= PQ_M$ and $NP^t = NQ_M^t|_{Y_M}$. Thus,
\begin{align}
\|PNP^t - \id_{Y_M}\| &=  \sup_{y\in Y_M, |y| = 1} |PQ_MNQ_M^ty - Py| \\
& \leq \sup_{y\in Y_M, |y| = 1} |Q_MNQ_M^ty - y| \\
&\leq |Q_MNQ_M^t - \id_{Y_M^{DG}}| \overset{\eqref{e_dizdar_operator_estimate}}{\leq} \frac{C}{K^2}.
\end{align}
\end{proof}
A consequence of Lemma~\ref{p_invertible_PNPt} is the following orthogonal decomposition of the space~$X_N = \mathbb{R}^N$. 
\begin{definition} \label{d_decomposition_fluctuation_mesoscopic_profile}
   Given~$x \in X_N$, let $x_{\parallel}$ denote the projection of~$x$ onto~$\ker P$ and let~$x_{\perp}$ denote the projection onto~$(\ker P)^{\perp} = \Image NP^t$. They are given by
   \begin{align*}
   x_{\parallel} = x - x_{\perp}  \quad  \mbox{and} \quad    x_{\perp} = NP^t(PNP^t)^{-1} P x.
   \end{align*} 
  \end{definition}

Lemma~\ref{p_fluctuation_gradient} then follows from this decomposition.

\section{Proof of Theorem~\ref{p_convergence_meso_to_macro_free_energy}}\label{s_convergence_free_energies}

This section is organized as follows. We first discuss the main idea of the proof. In Subsection~\ref{s_local_cramer_aux_results} we state and deduce several auxiliary results. We prove Theorem~\ref{p_convergence_meso_to_macro_free_energy} in Section~\ref{s_local_cramer_proof}. \\

Our goal is to show that the gradient of the coarse-grained Hamiltonian~$\bar H_{Y_M} : Y_M \to \mathbb{R}$ (cf.~Definition~\eqref{e_def_coarse_grained_Hmailtonian})
\begin{align}
   \bar H(y) := \bar H_{Y_M}(y) :=  - \frac{1}{N} \ln \int_{\left\{x \in X_N  :  Px =y \right\}} \exp \left( - H(x) \right) dx,
\end{align}
is close to that of the macroscopic free energy~$\mathcal H : L^2(\mathbb{T}) \to \mathbb{R}$ \begin{align}
  \mathcal{H}(y)  = \int_0^1 \varphi (y(\theta)) d \theta &\overset{\eqref{average_hydro_potential}}{=} \int_0^1 \sup_{\hat \theta \in \mathbb{R}} \left(\hat \theta y(\theta)  - \psi^*(\hat \theta) \right) \\
  &= \sup_{\hat y \in L^2(\mathbb{T})} \Big( \langle y, \hat y \rangle_{L^2} - \underbrace{\int_0^1 \psi^*(\hat y (\theta)) d \theta}_{=: \varphi^*(\hat y)} \Big) \label{e_def_macroscopic_free_energy}
\end{align}
where the supremum is attained by $\hat y = \varphi'(y)$.  \medskip

As in the proof of Theorem~\ref{p_strict_convexity_coarse_grained_Hamiltonian}, we will reduce the statement from the level of spline functions~$Y_M$ to the level of discontinuous Galerkin functions~$Y_{M}^{DG}$ (see Definition~\ref{d_Y_pol_and_H}). For this purpose we recall Definition~\ref{d_H_DG} of the coarse-grained Hamiltonian~$\bar H_{Y_M^{DG}}: Y_M^{DG} \to \mathbb{R}$: 
\begin{align}  \label{e_def_coarse_grained_wrt_Q}
  \bar H_{Y_M^{DG}} (y)=  - \frac{1}{N} \ln \int_{\left\{x \in \mathbb{R}^N :  Q_M x =y \right\}} \exp \left( - H(x) \right) \mathcal{L}^{ M}(dx)
\end{align}
where~$Q_M$ denotes the~$L^2$-orthogonal projection onto $Y_{M}^{DG}$.\medskip

In addition, we introduce the mesoscopic free energy $\mathcal{H}_{Y_M^{DG}}$ on the space of discontinuous Galerkin functions $Y_M^{DG}$:

\begin{definition}\label{d_meso_free_energy_DG}
  Let $\mathcal{H}_{Y_M^{DG}} : Y_M^{DG} \to \mathbb{R}$ be the function given by
\begin{align}
  \label{e_def_meso_free_energy}
   \mathcal{H}_{Y_M^{DG}} (z) : = \sup_{\hat z \in Y_M^{DG}} \left( \langle z, \hat z \rangle_{L^2} - \varphi_{N}^* (\hat z) \right),
\end{align}
where $\varphi_N^*: Y_M^{DG} \rightarrow \mathbb{R}$ is the function given by
\begin{align} \label{e_def_varphi_N_star}
  \varphi_{N}^* (\hat z) := \frac{1}{N} \sum_{i=1}^N \psi^*\left(N \int_{\frac{i-1}{N}}^{\frac{i}{N}} \hat z(s) d s\right) \overset{\eqref{barpsi*_def}}{=} \frac{1}{M}\sum_{m=1}^M \bar{\psi}^*_K(\hat z^{(m)})
\end{align}
where~$\hat z^{(m)} \in Y_1^{DG}, m =1, \cdots, M,$ are obtained by restricting $\hat z \in Y_M^{DG}$ to subintervals (cf. \eqref{e_restriction_to_subinterval}).
\end{definition}

The main ingredient of the proof is Lemma~\ref{e_convergence_cgH_Q_to_mfe_DIZDAR} below which states that~$\bar H_{Y_M^{DG}}$ is close to~$\mathcal{H}_{Y_M^{DG}}$ in~$C^2$ if~$K$ is large. This has been essentially established in Section~\ref{s_convexification}. The rest of the proof consists of arguing that~$\bar H$ is close to~$\bar H_{Y_M^{DG}}$ and~$\mathcal{H}$ is close to $\mathcal{H}_{Y_M^{DG}}$. The fact that ~$\bar H$ is close to a shifted version of~$\bar H_{Y_M^{DG}}$ follows from the formula \eqref{e_bar_h_Y_in_terms_of_bar_h_Y^DG} 
\begin{align*}
    \bar H (y) = -\frac{1}{N} \ln \int_{Y_M^\perp} \exp(- N \bar H_{Y_M^{DG}} (y+z)) \mathcal{L}^{LM}(dz) N^{LM}
\end{align*}
and the fact that~$\bar H_{Y_M^{DG}}$ is uniformly strictly convex (see Lemma~\ref{p_convexity_bar_H_Z}), and therefore the integral on the right hand side concentrates more and more around the minimum of~$\bar H_{Y_M^{DG}}$ for large~$K$. The fact that~$\mathcal{H}$ is close to $\mathcal{H}_{Y_M^{DG}}$ follows from the observation that as~$N \to \infty$ the function~$\varphi_{N}^*$ given by~\eqref{e_def_varphi_N_star} converges to the function~$\varphi^*$ and that as~$M \to \infty$, the spline space~$Y_M\subset L^2(\mathbb{T})$ approximates the full space~$L^2(\mathbb{T})$.

\subsection{Auxiliary results} \label{s_local_cramer_aux_results}

The first auxiliary result is that~$\bar H_{Y_M^{DG}}$ converges to $\mathcal{H}_{Y_M^{DG}}$ in $C^2$ as ~$K \to \infty$:

\begin{lemma}\label{e_convergence_cgH_Q_to_mfe_DIZDAR}
 There exists $K^*$ such that for $K \geq K^*$ and for all $M$ and ~$z \in Y_M^{DG}$,
  \begin{align}
    \label{e_cgHwrtZ_to_mfe_C_0}
    \left| \bar H_{Y_M^{DG}} (z) - \mathcal{H}_{Y_M^{DG}}(z)    \right| &\lesssim \frac{1}{K}, \\
    \label{e_cgHwrtZ_to_mfe_C_1}
    \left\| \nabla \bar H_{Y_M^{DG}} (z) - \nabla \mathcal{H}_{Y_M^{DG}}(z)    \right\| &\lesssim \frac{1}{K}, \\
    \label{e_cgHwrtZ_to_mfe_C_2}
    \| \Hess \bar H_{Y_M^{DG}} (z) - \Hess \mathcal H_{Y_M^{DG}}(z)    \| &\lesssim \frac{1}{K}.
  \end{align}
\end{lemma}

\begin{proof}[Proof of Lemma~\ref{e_convergence_cgH_Q_to_mfe_DIZDAR}]
By Lemma~\ref{p_decomp_bar_H_DG} and Definition~\ref{d_meso_free_energy_DG},
\begin{align}
  \bar H_{Y_M^{DG}}(z) = \frac{1}{M}\sum_{m=1}^M \psi_K(z^{(m)}), \quad
  \mathcal{H}_{Y_M^{DG}}(z) = \frac{1}{M}\sum_{m=1}^M \bar{\psi}_K( z^{(m)}).
\end{align}
Taking into account the different Euclidean structures on ~$Y_{M}^{DG}$ and~$ Y_1^{DG}$  as in the proof of Theorem~\ref{p_strict_convexity_bar_H_DG}, we see that the estimate~\eqref{e_cgHwrtZ_to_mfe_C_0} follows from~\eqref{orderzero}, the estimate~\eqref{e_cgHwrtZ_to_mfe_C_1} follows from~\eqref{e_local_cramer_convergence_gradient} and the estimate~\eqref{e_cgHwrtZ_to_mfe_C_2} follows from~\eqref{e_local_cramer_convergence_hessian}. 
\end{proof}

The next auxiliary result is that the coarse-grained Hamiltonians~$\bar H_{Y_M}$ and~$\bar H_{Y_M^{DG}}$ and the free energies~$\mathcal{H}_{Y_M^{DG}}$ and~$\mathcal{H}$ are uniformly strictly convex. We summarize those results in the following lemma.
\begin{lemma}\label{p_convexity_bar_H_Z}
There are constants~$0< \lambda < \Lambda < \infty$ and~$K_0$ such that if~$K\geq K_0$ then for all~$z \in Y_M^{DG}$
\begin{align}
  \label{e_convexity_bar_H_Z}
  \lambda \Id_{Y_M^{DG}} &\leq \Hess \bar H_{Y_M^{DG}} (z) \leq \Lambda \Id_{Y_M^{DG}} \\
  \label{e_convexity_meso_fe_on_Z}
  \lambda \Id_{Y_M^{DG}} &\leq \Hess \mathcal{H}_{Y_M^{DG}} (z) \leq \Lambda \Id_{Y_M^{DG}}. \\
  \label{e_convexity_varphi_N}
  \lambda \Id_{Y_M^{DG}} &\leq \Hess \varphi_N^* (z) \leq \Lambda \Id_{Y_M^{DG}}.
\end{align}
Under the same conditions, for all~$z \in Y_M$
\begin{align}
  \label{e_convexity_bar_H_Y}
  \lambda \Id_{Y_M} \leq \Hess \bar H_{Y_M} (z) \leq \Lambda \Id_{Y_M}.\end{align}
Finally, for all~$z \in L^2$
\begin{align}
  \label{e_convexity_bounds_varphi}
  \lambda \Id_{L^2} &\leq \Hess \varphi^* (z) \leq \Lambda \Id_{L^2}\\
  \label{e_convexity_mathcal_H}
  \lambda \Id_{L^2} &\leq \Hess \mathcal{H}  (z) \leq \Lambda \Id_{L^2}.
\end{align}
All inequalities are in the sense of quadratic forms. 
\end{lemma}

\begin{proof}[Proof of Lemma~\ref{p_convexity_bar_H_Z}]
The estimate~\eqref{e_convexity_bar_H_Z} is given by Theorem~\ref{p_strict_convexity_bar_H_DG}. The estimate \eqref{e_convexity_varphi_N} follows from \eqref{bdhesspsiJ*}, from which the estimate~\eqref{e_convexity_meso_fe_on_Z} follows by basic properties of Legendre transform. The estimate~\eqref{e_convexity_bar_H_Y} is given by Theorem~\ref{p_strict_convexity_coarse_grained_Hamiltonian}. The estimate ~\eqref{e_convexity_bounds_varphi} follows from the uniform strict convexity of $\psi^*$ and uniform bound of $(\psi^*)''$ (cf. \eqref{boundvar}) since
\begin{equation}
\langle y_1, \varphi^*(x) y_2\rangle_{L^2} = \int y_1(\theta)y_2(\theta) (\psi^*)''(\theta) d\theta.
\end{equation} 
Similarly, the estimate ~\eqref{e_convexity_mathcal_H} follow from the uniform strict convexity of $\varphi$ and uniform bound of $\varphi''$ ($\varphi$ is the Legendre transform of $\psi^*$). 
\end{proof}

The next auxiliary statement shows a nice relation between the Hamiltonians~$\bar H_{Y_M}$ and~$\bar H_{Y_M^{DG}}$. Recall that $Y_M^\perp := \left\{ z \in Y_M^{DG}  :  Pz =0 \right\}$. 
\begin{lemma}\label{p_error_laplace_method}
For every ~$y \in Y_M$, there exists a unique~$\bar z^* \in Y_M^\perp$ such that
  \begin{align}\label{e_def_z_star}
    \bar H_{Y_M^{DG}} (y+ \bar z^*) = \inf_{z\in Y_M^\perp} \bar H_{Y_M^{DG}} (y+ z).
  \end{align}
Then for any~$y \in Y_M$,
  \begin{equation}\label{e_z_star_property}
  \nabla \bar H_{Y_M^{DG}} (y+ \bar z^*) = P \nabla \bar H_{Y_M^{DG}} (y+ \bar z^*).
  \end{equation}
  and
\begin{align}
  \label{e_error_laplace_method}
  | \nabla \bar H(y) - P\nabla \bar H_{Y_M^{DG}}(y + \bar z^*) |_{L^{2}} \lesssim \frac{1}{K}.
\end{align}
\end{lemma}
\begin{proof}[Proof of Lemma~\ref{p_error_laplace_method}]
  The statement~\eqref{e_def_z_star} follows from the strict convexity ~\eqref{e_convexity_bar_H_Z} of~$\bar H_{Y_M^{DG}}$, applied to the affine subspace $y+Y_M^\perp$. The statement \eqref{e_z_star_property} follows directly from ~\eqref{e_def_z_star}.
  
Let us now turn to the verification of~\eqref{e_error_laplace_method}. It follows from~\eqref{e_bar_h_Y_in_terms_of_bar_h_Y^DG} that
\begin{align}
 \nabla \bar H (y) = \frac{\int_{Y_M^\perp} P\nabla \bar H_{Y_M^{DG}} (y+ z)  \exp(- N \bar H_{Y_M^{DG}} (y+z)) dz}{\int_{Y_M^\perp} \exp(- N \bar H_{Y_M^{DG}}) (y+z) dz}.
\end{align}
From this, we get
\begin{align}
&  | \nabla \bar H (y) - P\nabla \bar H_{Y_M^{DG}} (y+ \bar z^*)|^2\\
 & = \left| \frac{\int_{Y_M^\perp} \left(  P\nabla \bar H_{Y_M^{DG}} (y+ z)   - P\nabla \bar H_{Y_M^{DG}} (y+ \bar z^*) \right) e^{- N \bar H_{Y_M^{DG}} (y+z) } dz}{\int_{Y_M^\perp} \exp(- N \bar H_{Y_M^{DG}} (y+z)) dz} \right|^2  \\
& \leq \frac{\int_{Y_M^\perp} \left|  \nabla \bar H_{Y_M^{DG}} (y+ z)   - \nabla \bar H_{Y_M^{DG}} (y+ \bar z^*) \right|^2 e^{- N \bar H_{Y_M^{DG}} (y+z) } dz }{\int_{Y_M^\perp} \exp(- N \bar H_{Y_M^{DG}} (y+z)) dz} \\
& \quad \leq \Lambda^2 \ \frac{\int_{Y_M^\perp} \left|  z -\bar z^* \right|^2 \exp(- N \bar H_{Y_M^{DG}} (y+z)) dz }{\int_{Y_M^\perp} \exp(- N \bar H_{Y_M^{DG}} (y+z)) dz},  \label{e_laplace_error_bound_setp_1}
\end{align}
where we used the upper bound ~\eqref{e_convexity_bar_H_Z} on ~$\operatorname{Hess} \bar H_{Y_M^{DG}}$ in the last step. Now, using the convexity bound ~\eqref{e_convexity_bar_H_Z} of ~$\bar H_{Y_M^{DG}}$, we get
\begin{align}
   \left|  z - \bar z^* \right|^2  & \leq \frac{1}{\lambda} \left(z - \bar z^* \right) \cdot \left( \nabla \bar H_{Y_M^{DG}} (y+ z) - \nabla \bar H_{Y_M^{DG}} (y + \bar z^*) \right)  \\
& \overset{\eqref{e_z_star_property}}{=}  \frac{1}{\lambda} \left(z - \bar z^* \right)  \cdot \nabla \bar H_{Y_M^{DG}} (y+ z) .  \label{e_laplace_error_bound_setp_2}
\end{align}
Inserting the estimate~\eqref{e_laplace_error_bound_setp_2} into the right hand side of~\eqref{e_laplace_error_bound_setp_1}, and using integration by parts, we get
\begin{align}
&  | \nabla \bar H (y) - P\nabla \bar H_{Y_M^{DG}} (y+ \bar z^*)|^2  \\
 & \leq \frac{\Lambda^2}{\lambda} \ \frac{\int_{Y_M^\perp} \left(z - \bar z^* \right)  \cdot \nabla \bar H_{Y_M^{DG}} (y+ z) \exp(- N \bar H_{Y_M^{DG}} (y+z)) \mathcal{L}^{LM}(dz) }{\int_{Y_M^\perp} \exp(- N \bar H_{Y_M^{DG}} (y+z)) \mathcal{L}^{LM}(dz)}\\
& = -\frac{\Lambda^2}{\lambda N} \ \frac{\int_{Y_M^\perp} \left(z - \bar z^* \right)  \cdot \nabla \left[ \exp(- N \bar H_{Y_M^{DG}} (y+z)) \right] \mathcal{L}^{LM}(dz) }{\int_{Y_M^\perp} \exp(- N \bar H_{Y_M^{DG}} (y+z)) \mathcal{L}^{LM}(dz)} \\
& = \frac{\Lambda^2}{\lambda N} \ \frac{\int_{Y_M^\perp} \nabla \cdot \left(z - \bar z^* \right)   \exp(- N \bar H_{Y_M^{DG}} (y+z)) \mathcal{L}^{LM}(dz) }{\int_{Y_M^\perp} \exp(- N \bar H_{Y_M^{DG}} (y+z)) \mathcal{L}^{LM}(dz)} \\
&= \frac{\Lambda^2}{\lambda N } \dim Y_M^\perp = \frac{\Lambda^2}{\lambda} \ \frac{L}{K},
\end{align}
which is the desired estimate~\eqref{e_error_laplace_method}.
\end{proof}

Let us introduce the mesoscopic free energy~$\mathcal{H}_{Y_M}: Y_M \to \mathbb{R}$ that is associated to the spline space~$Y_M$.
\begin{definition}
  Let~$\mathcal{H}_{Y_M}  : Y_M \to \mathbb{R}$ be the function given by
\begin{align}
  \label{e_d_meso_fe_Y}
  \mathcal{H}_{Y_M} (y) = \sup_{\hat y \in Y_M} \left( \langle y, \hat y \rangle_{L^2} - \varphi_{N}^*(\hat y) \right),
\end{align}
where~$\varphi_{N}^* (\hat y)$ is given by~\eqref{e_def_varphi_N_star}. 
\end{definition}

The next auxiliary statement shows a nice relation between the mesoscopic free energies~$\mathcal{H}_{Y_M}$ and~$\mathcal{H}_{Y_M^{DG}}$. 
\begin{lemma}\label{p_relation_meso_fe_Y_and_Z}
For every~$y \in Y_M$, there exists a unique~$ z^* \in Y_M^\perp$ such that
\begin{align*}
  \inf_{Y_M^\perp} \mathcal{H}_{Y_M^{DG}}(y +z)  =\mathcal{H}_{Y_M^{DG}}(y + z^*) .
\end{align*}
Then for all $y\in Y_M$,
  \begin{align}\label{e_relation_meso_fe_Y_and_Z}
 \mathcal{H}_{Y_M} (y) = \mathcal{H}_{Y_M^{DG}}(y + z^*) = \inf_{Y_M^\perp} \mathcal{H}_{Y_M^{DG}}(y +z)
  \end{align}
and therefore,
  \begin{align}\label{e_relation_gradient_meso_fe_Y_and_Z}
\nabla \mathcal{H}_{Y_M} (y) = \nabla  \mathcal{H}_{Y_M^{DG}}(y + z^*) = P \nabla  \mathcal{H}_{Y_M^{DG}}(y + z^*).
  \end{align}
Let~$\bar  z^*$ be as in \eqref{e_def_z_star}, then for all ~$y \in Y_M$,
\begin{align}
  \label{e_relation_between_z_star_and_hat_z}
  |\bar z^* - z^*|_{L^2} \lesssim  \frac{1}{K},
\end{align}
and therefore by~\eqref{e_convexity_bar_H_Z}
\begin{align}\label{e_estimate_bar_h_Z_hat_z_star_Z}
  |\nabla \bar H_{Y_{M}^{DG}} (y + \bar z^*) - \nabla \bar H_{Y_{M}^{DG}} (y + z^*)|_{L^2}
  \lesssim  \frac{1}{K}.  
\end{align}
\end{lemma}
\begin{proof}[Proof of Lemma~\ref{p_relation_meso_fe_Y_and_Z}]
The unique existence of $z^*$ follows directly from the strict convexity of~$\mathcal{H}_{Y_M^{DG}}$. Let us prove \eqref{e_relation_meso_fe_Y_and_Z}. Notice that in general for a function $F$
\begin{equation}
 \inf_x \sup_y F(x, y) \geq \sup_y \inf_x F(x, y).
\end{equation}
Using this, we get
\begin{align}
  \inf_{ z\in Y_M^\perp}   \mathcal{H}_{Y_M^{DG}}(y+ z) 
  &\overset{\eqref{e_def_meso_free_energy}}{=} \inf_{ z\in Y_M^\perp}  \sup_{\hat z \in Y_M^{DG}} \left( \langle y+ z , \hat z \rangle_{L^2} - \varphi_{N}^* (\hat z) \right) \\
 &\geq \sup_{\hat z \in Y_M^{DG}} \inf_{  z\in Y_M^\perp}  \left( \langle y+ z, \hat z \rangle_{L^2} - \varphi_{N}^* (\hat z) \right) \\
 &= \sup_{\hat z \in Y_M^{DG}}   \left[ \left( \langle y, \hat z \rangle_{L^2} - \varphi_{N}^* (\hat z) \right)    +  \inf_{ z\ in Y_M^\perp }  \langle z, \hat z \rangle_{L^2} \right].
\end{align}
We observe that
\begin{align*}
  \inf_{  z\in Y_M^\perp} \langle z, \hat z \rangle_{L^2} =
  \begin{cases}
    0, & \mbox{if } \hat z \in Y_M \\
    - \infty, & \mbox{if } \hat z \notin Y_M.
  \end{cases}
\end{align*}
Hence, we get that 
\begin{align}
  \label{e_e_relation_meso_fe_Y_and_Z_first_identity_2}
  \inf_{ z\in Y_M^\perp}   \mathcal{H}_{Y_M^{DG}}(y+ z) 
  &\geq \sup_{\hat z \in Y_M^{DG}}    \left[ \left( \langle y, \hat z \rangle_{L^2} - \varphi_{N}^* (\hat z) \right)    +  \inf_{  z\in Y_M^\perp }  \langle z, \hat z \rangle_{L^2} \right] \\
& = \sup_{\hat z \in Y_M}     \left( \langle y, \hat z \rangle_{L^2} - \varphi_{N}^* (\hat z) \right) \overset{\eqref{e_d_meso_fe_Y}}{=} \mathcal{H}_{Y_M} (y).
\end{align}

It remains to show the opposite inequality.
By strict convexity of $\varphi_N^*$, there exists $\hat z^{max} \in Y_M^{DG}$ such that
\begin{equation}
\mathcal{H}_{Y_M^{DG}}(y+ z^*) = \langle y+ z^* , \hat z^{max} \rangle_{L^2} - \varphi_{N}^* (\hat z^{max}).
\end{equation}
By basic properties of Legendre transform, 
\begin{equation}
\hat z^{max} = \nabla \mathcal{H}_{Y_M^{DG}}(y+ z^*) \overset{\eqref{e_relation_gradient_meso_fe_Y_and_Z}}{=} P\nabla \mathcal{H}_{Y_M^{DG}}(y+ z^*) \in Y_M.
\end{equation}
Hence, $\hat z^{max} \perp z^*$, and
\begin{align}
\mathcal{H}_{Y_M^{DG}}(y+ z^*) &= \langle y+ z^* , \hat z^{max} \rangle_{L^2} - \varphi_{N}^* (\hat z^{max})\\
&= \langle y , \hat z^{max} \rangle_{L^2} - \varphi_{N}^* (\hat z^{max})\\
& \leq  \sup_{\hat z \in Y_M} \left( \langle y , \hat z \rangle_{L^2} - \varphi_{N}^* (\hat z) \right)\overset{\eqref{e_d_meso_fe_Y}}{=} \mathcal{H}_{Y_M} (y).
\end{align}
This completes the verification of~\eqref{e_relation_meso_fe_Y_and_Z}.\\

Let us now turn to the verification of~\eqref{e_relation_gradient_meso_fe_Y_and_Z}. The second equality directly follows from the definition of $z^*$. he main observation for proving the first equality is that $z^*: Y_M \rightarrow Y_M^\perp$ is a $C^1$ map. To verify the smoothness, let $P_{Y_M^\perp}: Y_M^{DG} \rightarrow Y_M^\perp$ denote the $L^2$ orthogonal projection onto $Y_M^\perp$, and let $F: Y_M\oplus Y_M^\perp \rightarrow Y_M^\perp$ be the function given by $F(y, z) = P_{Y_M^\perp} \nabla \mathcal{H}_{Y_M^{DG}}(y+z)$. Then
\begin{align}
& F(y, z^*(y)) = 0, \\
& F \mbox{ is } C^1 \quad \mbox{and} \quad \frac{\partial F}{\partial z} = \operatorname{Hess} \mathcal{H}_{Y_M^{DG}}|_{Y_M^\perp}.
\end{align} 
Since $\operatorname{Hess} \mathcal{H}_{Y_M^{DG}}$ is positive definite by strict convexity of $\mathcal{H}_{Y_M^{DG}}$,  $\frac{\partial F}{\partial z}$ is also positive definite and so invertible, which shows $z^*$ is $C^1$ by implicit function theorem. Now, the first equality of \eqref{e_relation_gradient_meso_fe_Y_and_Z} follows from taking gradient of both sides of the first equality of \eqref{e_relation_meso_fe_Y_and_Z} by applying the chain rule.\\

Let us turn to the verification of~\eqref{e_relation_between_z_star_and_hat_z}. We observe that due to the definition of~$\bar z^*$ and~$z^*$ it holds
\begin{align}
  \label{e_nabla_h_0_at_z}
  \nabla \bar H_{Y_M^{DG}} (y+ \bar z^*) \in Y_M, \quad \mbox{and} \quad \nabla \mathcal{H}_{Y_M^{DG}} (y+ z^*) \in Y_M.
\end{align}
Using this and the convexity bound~\eqref{e_convexity_bar_H_Z} of~$\bar H_{Y_M^{DG}}$, we get that
\begin{align*}
  | \bar z^* - z^*|_{L^2}^2 & \leq \frac{1}{\lambda} \langle  \nabla \bar H_{Y_M^{DG}}(y+ \bar z^*) - \nabla \bar H_{Y_M^{DG}} (y+ z^*) , \bar z^* - z^* \rangle_{L^{2}} \\
& = - \frac{1}{\lambda} \langle   \nabla \bar H_{Y_M^{DG}} (y+ z^*) , \bar z^* - z^* \rangle_{L^{2}} \\ 
& = \frac{1}{\lambda} \langle   \nabla \mathcal{H}_{Y_M^{DG}} (y+ z^*) - \nabla \bar H_{Y_M^{DG}} (y+ z^*) , \bar z^* - z^* \rangle_{L^{2}} \\ 
& \overset{\eqref{e_cgHwrtZ_to_mfe_C_1}}{\lesssim} \frac{1}{\lambda} \ \frac{1}{K} |\bar z^* - z^*|_{L^2},
\end{align*}
which yields the desired estimate~\eqref{e_relation_between_z_star_and_hat_z}.
\end{proof}

The last auxiliary result shows that~$\nabla \mathcal{H}_{Y_M}$ and~$\nabla \mathcal{H}$ are close. 
\begin{lemma}\label{p_numerical_error}
It holds that for any~$x \in L^2$
\begin{align}
    \label{e_numerical_error}
    |  \nabla \mathcal{H}_{Y_M}(Px) - \nabla \mathcal{H}(x)  |_{L^2}  \lesssim  \left(\frac{1}{K} + \frac{1}{M}\right)(|x|_{L^2}+|x|_{H^1}).
  \end{align}
\end{lemma}

\begin{proof}[Proof of Lemma~\ref{p_numerical_error}]
We introduce another mesoscopic free energy on the spline space $Y_M$. Let $\hat{\mathcal{H}}_{Y_M}: Y_M\to \mathbb{R}$ be the function given by
  \begin{align}
    \label{e_def_auxiliary_meso_fe_numerical_error}
   \hat{\mathcal{H}}_{Y_M} (y) = \sup_{\hat y \in Y_M} \left( \langle y , \hat y \rangle_{L^2} - \varphi^*(\hat y) \right),
  \end{align}
where~$\varphi^*(\hat y)$ is defined in~\eqref{e_def_macroscopic_free_energy}. 
For~$x \in L^2$ let us set~$y=Px$. By basic properties of the Legendre transform,
\begin{alignat}{4}
  \label{e_def_hat_y_N}
  \nabla \mathcal{H}_{Y_M} (y) &= \hat y_N \quad &&\mbox{where } \hat y_N \in Y_M &&\mbox{ and } P \left( \nabla \varphi_N^* \right) (\hat y_N) &&=y,\\
 \label{e_def_hat_y}
  \nabla \hat{\mathcal{H}}_{Y_M} (y) &= \hat y \quad &&\mbox{where } \hat y \in Y_M &&\mbox{ and } \quad P \left( \nabla \varphi^* \right) (\hat y) &&= y ,\\
 \label{e_def_hat_x}
  \nabla \mathcal{H} (x) &= \hat x \quad &&\mbox{where } \hat x \in L^{2} &&\mbox{ and } \qquad \left( \nabla \varphi^* \right) (\hat x) &&=x.
\end{alignat}
Below, we will show that
\begin{align}
  \label{e_numerical_error_est_hat_y_hat_y_N}
  | \hat y_N - \hat y |_{L^2} &\lesssim   \frac{1}{K} \ |x|_{L^2}, \\
  \label{e_numerical_error_est_hat_y_hat_x}
  | \hat y - \hat x |_{L^2} &\lesssim   \frac{1}{M} \ |x|_{H^{1}},
\end{align}
from which the desired inequality~\eqref{e_numerical_error} follows by triangle inequality. We start with deducing~\eqref{e_numerical_error_est_hat_y_hat_y_N}. By \eqref{e_def_hat_y_N} and~\eqref{e_def_hat_y},
 \begin{align}
 \langle \nabla \varphi_N^*(\hat y_N) - \nabla \varphi^* (\hat y) , \hat y_N - \hat y \rangle_{L^2} &= \langle  P \left(\nabla \varphi_N^*(\hat y_N) - \nabla \varphi^* (\hat y)\right) , \hat y_N - \hat y \rangle_{L^2}\\
  &= \langle  y -y , \hat y_N - \hat y \rangle_{L^2} = 0.
 \end{align}
 Therefore, using the convexity bounds~\eqref{e_convexity_bounds_varphi} of~$\varphi^*$, we get
\begin{align}
   \lambda | \hat y_N - \hat y|_{L^2}^2 &\leq  \langle \nabla \varphi^*(\hat y_N) - \nabla \varphi^* (\hat y) , \hat y_N - \hat y \rangle_{L^2} \\
   &= \langle \nabla \varphi^*(\hat y_N) - \nabla \varphi_N^* (\hat y_N) , \hat y_N - \hat y \rangle_{L^2} \\
& \leq  | \nabla \varphi^*(\hat y_N) - \nabla \varphi_N^*(\hat y_N) |_{L^2}    |\hat y_N - \hat y|_{L^2}.
\label{e_numerical_error_est_hat_y_hat_x_step_1}
\end{align}
Using the definitions ~\eqref{e_def_macroscopic_free_energy} and~\eqref{e_def_varphi_N_star} of~$\varphi^*$ and~$\varphi_N^*$, we find
\begin{eqnarray*}  \lefteqn{| \nabla \varphi^*(\hat y_N) - \nabla \varphi_N^*(\hat y_N) |_{L^2}^2} \\
 & =&  \sum_{i=1}^N  \int_{\frac{i-1}{N}}^{\frac{i}{N}} \left|(\psi^{*})'(\hat y_N (\theta)) - (\psi^{*})' \left( N \int_{\frac{i-1}{N}}^{\frac{i}{n}} \hat y_N (s) ds \right) \right|^2 d \theta   \\
& \overset{\eqref{boundvar}}{\lesssim}&   \sum_{i=1}^N \int_{\frac{i-1}{N}}^{\frac{i}{N}} \left| \hat y_N (\theta) -  N \int_{\frac{i-1}{N}}^{\frac{i}{n}} \hat y_N (s) ds \right|^2 d \theta \\
& \lesssim&   \frac{1}{N^2} \sum_{i=1}^N \int_{\frac{i-1}{N}}^{\frac{i}{N}} \left| \hat y_N' (\theta)  \right|^2 d \theta    \\
&=&   \frac{1}{N^2}   \left| \hat y_N \right|_{H^1}^2  \overset{\eqref{e_inverse_sobolev}}{\lesssim}   \frac{M^2}{N^2}   \left| \hat y_N \right|_{L^2}^2  \overset{\eqref{e_def_hat_y_N}}{=}  \frac{1}{K^2}   \left| \nabla \mathcal{H}_{Y_M}(y) \right|_{L^2}^2
\end{eqnarray*}
where we used a Poincar\'e inequality on an interval of length~$\frac{1}{M}$ and then an inverse Sobolev inequality \eqref{e_inverse_sobolev} from below. Let $y_0$ be the minimizer of $\mathcal{H}_{Y_M}$, then $\nabla \mathcal{H}_{Y_M} (y_0) = 0$, and
\begin{align} y_0 \overset{\eqref{e_def_hat_y_N}}{=} P \nabla \varphi_N^* (0) 
= P (\psi^*)'(0)
= (\psi^*)'(0) 
\overset{\eqref{e_ss_mean_0}}{=} 0.
\end{align}
Therefore, because $\Hess \mathcal{H}_{Y_M}$ is uniformly bounded (which follows from the uniform strict convexity of $\varphi^*_N$), we get
\begin{align}\left| \nabla \mathcal{H}_{Y_M}(y) \right|_{L^2}^2 =   \left| \nabla \mathcal{H}_{Y_M}(y) - \nabla \mathcal{H}_{Y_M}(0) \right|_{L^2}^2 
\lesssim |y|_{L^2}^2 
\leq |x|_{L^2}^2.
\end{align}
Inserting this estimate into~\eqref{e_numerical_error_est_hat_y_hat_x_step_1} yields the desired estimate~\eqref{e_numerical_error_est_hat_y_hat_y_N}. \\

Let us now turn to the last missing ingredient of this argument, namely the estimate~\eqref{e_numerical_error_est_hat_y_hat_x}. By~\eqref{e_def_hat_y} and~\eqref{e_def_hat_x}, for all~$\zeta \in Y_M$
\begin{align}
  \langle \nabla \varphi^*(\hat y) -  \nabla \varphi^*(\hat x)  , \zeta \rangle_{L^2} & = \langle P \nabla \varphi^*(\hat y) - P \nabla \varphi^*(\hat x)  , \zeta \rangle_{L^2} \\
  & =\langle Px - P x  , \zeta \rangle_{L^2} 
 =0 .
\end{align}
Choosing $\zeta = \hat y - P \hat x = ( \hat y - \hat x ) + ( \hat x - P \hat x )$, we get by  using~\eqref{e_convexity_bounds_varphi} that
\begin{align}
  \lambda |\hat y - \hat x|_{L^2}^2 & \leq \langle \nabla \varphi^* (\hat y) - \nabla \varphi^* (\hat x) , \hat y - \hat x \rangle_{L^2} \\
& = - \langle \nabla \varphi^* (\hat y) - \nabla \varphi^* (\hat x) , \hat x - P \hat x \rangle_{L^2}  \\
& \leq |\nabla \varphi^* (\hat y) - \nabla \varphi^* (\hat x)|_{L^2} |\hat x - P \hat x|_{L^2} \\
& \leq \Lambda |\hat y - \hat x|_{L^2} |\hat x - P \hat x|_{L^2}. \label{e_numerical_error_est_hat_y_hat_x_step_2}
\end{align}
Finally, by ~\eqref{e_penalization_fluctuations_L_2_norm} from below,
\begin{align}
|\hat x - P \hat x|_{L^2} \lesssim   \frac{1}{M}  |\hat x|_{H^1}  =  \frac{1}{M} |\nabla \mathcal{H} (x)|_{H^1},
\end{align}
and by the uniform bound on $\varphi''$,
\begin{align}
|\nabla \mathcal{H} (x)|_{H^1} = |\varphi'(x)|_{H^1} &= |\partial_\theta \varphi'(x)|_{L^2} \\
&= |\varphi''(x) \partial_\theta x|_{L^2} \lesssim |\partial_\theta x|_{L^2} = |x|_{H^1}, 
\end{align}
which, combined with \eqref{e_numerical_error_est_hat_y_hat_x_step_2},
give the desired estimate~\eqref{e_numerical_error_est_hat_y_hat_x}.

\end{proof}

\subsection{Proof of Theorem~\ref{p_convergence_meso_to_macro_free_energy}}\label{s_local_cramer_proof}

Using the auxiliary results that were provided in Section~\ref{s_local_cramer_aux_results}, proving Theorem~\ref{p_convergence_meso_to_macro_free_energy} becomes straightforward.

\begin{proof}[Proof of Theorem~\ref{p_convergence_meso_to_macro_free_energy}]
	For any~$\zeta \in L^2(\mathbb{T})$ and~$y=P\zeta$,
	\begin{align}
	| \nabla \bar H (P\zeta) -  \nabla \mathcal{H}(\zeta)|_{L^2}  & \leq | \nabla \bar H (y) - P \nabla \bar H_{Y_M^{DG}} (y + \bar z^*) |_{L^2} \\
	&\quad  + |P \nabla \bar H_{Y_M^{DG}} (y + \bar z^*)  - P \nabla \bar H_{Y_M^{DG}} (y + z^*) |_{L^2} \\ 
	& \quad + |P \nabla \bar H_{Y_M^{DG}} (y + z^*) - P\nabla \mathcal{H}_{Y_M^{DG}} (y + z^*)  |_{L^2} \\
	&  \quad + |P \nabla \mathcal{H}_{Y_M^{DG}} (y + z^*) - \nabla \mathcal{H}(\zeta)  |_{L^2}.
	\end{align}
	The first term on the right hand side of the last estimate is estimated by~\eqref{e_error_laplace_method}. The second term is estimated by~\eqref{e_estimate_bar_h_Z_hat_z_star_Z}. The third term is estimated by~\eqref{e_cgHwrtZ_to_mfe_C_1}. And finally, the fourth term is estimated by using~\eqref{e_relation_gradient_meso_fe_Y_and_Z} and~\eqref{e_numerical_error}. Summing up yields the desired estimate~\eqref{e_estimate_meso_to_macro_free_energy}.
\end{proof}

\subsection{Properties of spline approximations}
In  the proof of Lemma~\ref{p_numerical_error} we used an inverse Sobolev inequality \eqref{e_inverse_sobolev} and an estimate~\eqref{e_penalization_fluctuations_L_2_norm} from below. Let us now state and prove these results as well as some related results on spline approximations that will be used in the companion article \cite{DMOW18a}. 
\begin{lemma}[Inverse Sobolev inequality] \label{p_inverse_sobolev}
	For all~$y \in Y_M$ holds
	\begin{align}\label{e_inverse_sobolev}
	|y|_{H^{2}} \lesssim M  |y|_{H^{1}} \lesssim M^2 |y|_{L^2}.
	\end{align}
\end{lemma}
It is clear that the estimate~\eqref{e_inverse_sobolev} holds. The spline spaces $Y_M$ are finite dimensional and norms on finite-dimensional vector spaces are equivalent. The factor~$M^2$ comes from a scaling argument i.e.~$\frac{1}{M}$ is the only internal length scale. We omit the details of the proof, which consists of a straight forward calculation.
\begin{lemma}[Approximation properties of splines]\label{p_penalization_fluctuations_H_neg_norm}
  Let~$\zeta \in L^{2}(\mathbb{T})$. For~$\zeta \in L^{2}(\mathbb{T})$ we define the spline interpolation~$I \zeta \in Y_M$ as
  \begin{align}\label{e_def_spline_interpolation}
    I \zeta (\theta) = \sum_{j=1}^M \zeta \left( \frac{2j-1}{2M} \right)  B_j(\theta),
  \end{align}
where~$B_j \in Y_M$ is the B-spline basis of $Y_M$ given by
\begin{equation} \label{bspline}
B_j(\theta) =
 \begin{cases}
  \frac{M^2}{2} (\theta-\frac {j-2}{M})^2 & \text{for}\, \theta\in [\frac{j-2}{M},\frac{j-1}{M}) \\
  \frac{3}{4} \,-\,M^2(\theta - \frac{2j-1}{2M})^2 & \text{for}\, \theta\in [\frac{j-1}{M},\frac{j}{M}) \\
  \frac{M^2}{2} (\theta-\frac {j+1}{M})^2 & \text{for}\, \theta \in [\frac{j}{M},\frac{j+1}{M}) \\
  0 & \text{else}.
 \end{cases}
\end{equation} 
Then 
  \begin{align}\label{e_penalization_fluctuations_L_2_norm}
       |\zeta-P\zeta |_{L^2}   \leq \left| \zeta - I \zeta \right|_{L^2} \lesssim \frac{1}{M} \left| \zeta \right|_{H^1}
  \end{align}
and
\begin{align}
  \label{e_control_spilne_H_1_norm}
  |P\zeta |_{H^1}   \lesssim |\zeta |_{H^1} .
\end{align}
\end{lemma}

\begin{proof}[Proof of Lemma~\ref{p_penalization_fluctuations_H_neg_norm}]
We start with deducing ~\eqref{e_penalization_fluctuations_L_2_norm}. The first estimate of \eqref{e_penalization_fluctuations_L_2_norm} is due to the best approximation property of $P$. The second estimate of \eqref{e_penalization_fluctuations_L_2_norm} is well known in the literature on~$B$-splines (see for example Theorem 7.3 in \cite{devore_lorentz_1993}). For the convenience of the reader we give a short proof of this fact. We exploit the fact that the $B_j$ have small support. We observe
\begin{equation}
\sum_{j=1}^M B_j \equiv 1 \label{partition1}.
\end{equation}

In this case, we obtain for $\theta \in \left(\frac{m-1}{M},\frac{m}{M}\right)$:
\begin{eqnarray*}
\zeta(\theta) - I \zeta  (\theta) &\stackrel{\eqref{bspline},\eqref{partition1}}{=}& \sum_{j=0}^2
 \left(\zeta (\theta) - \zeta\left(\frac{2m-3 +2j}{2M}\right)\right) B_{m-1+j}(\theta).
\end{eqnarray*}
We use Young's inequality, the Fundamental Theorem of Calculus and
the Cauchy-Schwarz inequality to deduce an estimate:
\begin{eqnarray*}
\lefteqn{\int_{\frac{m-1}{M}}^{\frac{m}{M}} |\zeta (\theta) - I \zeta (\theta)|^2 \,d\theta}\\
&\leq& 3  \sum_{j=0}^2 \int_{\frac{m-1}{M}}^{\frac{m}{M}} \left(\zeta (\theta) - \zeta\left(\frac{2m-3 +2j}{2M}\right)\right) B_{m-1+j}(\theta)\\
&\leq& \int_{\frac{m-1}{M}}^{\frac{m}{M}} 3\,\frac{2}{M}\,\left(\int_{\frac{2m-3}{2M}}^{\frac{2m+1}{2M}} |\zeta '(\widetilde{\theta})|^2
d\widetilde{\theta}\right) \; (B_{m-1}^2(\theta)+B_{m}^2(\theta)+B_{m+1}^2(\theta)) \,d\theta\\
&\leq& \frac{6}{M^2} \left|\sum_{j=1}^M B_j^2 \right|_{L^{\infty}} \,\left(\int_{\frac{2m-3}{2M}}^{\frac{2m+1}{2M}} |\zeta '(\widetilde{\theta})|^2 d\widetilde{\theta}\right).
\end{eqnarray*}
Adding up all inequalities for $m=1,...,M$ yields the second estimate of \eqref{e_penalization_fluctuations_L_2_norm}. \smallskip

Let us now turn to the verification of~\eqref{e_control_spilne_H_1_norm}. It suffices to show that
\begin{align}\label{e_aux_estimates_for_tianqi_estimate}
  |\zeta - I \zeta|_{H^1} \lesssim |\zeta|_{H^1} \qquad \mbox{and} \qquad |P \zeta - I \zeta|_{H^1} \lesssim |\zeta|_{H^1}
\end{align}
from which we get
\begin{align}
  |P \zeta|_{H^1} \leq |P \zeta - I \zeta|_{H^1} + |I \zeta - \zeta|_{H^1} + |\zeta|_{H^1}  \lesssim |\zeta|_{H^1}.
\end{align}
The first estimate of~\eqref{e_aux_estimates_for_tianqi_estimate} can be deduced with similar calculations as were used for verifying the second estimate of \eqref{e_penalization_fluctuations_L_2_norm}. We leave the details as an exercise. The second estimate of~\eqref{e_aux_estimates_for_tianqi_estimate} follows from the inverse Sobolev inequality on $Y_M$ and the estimates of \eqref{e_penalization_fluctuations_L_2_norm}:
\begin{align}
  |P \zeta - I \zeta |_{H^1} &\overset{\eqref{e_inverse_sobolev}}{\lesssim} M |P \zeta - I \zeta|_{L^2} \\
  & \leq M(|P\zeta - \zeta|_{L^2} +  |\zeta - I \zeta|_{L^2} )
 \overset{\eqref{e_penalization_fluctuations_L_2_norm}}{\lesssim  }  |\zeta|_{H^1}.
\end{align}
\end{proof}

\section*{Acknowledgment}
This research has been partially supported by NSF grant DMS-1407558. Georg Menz and Tianqi Wu want to thank the Max-Planck Institute for Mathematics in the Sciences, Leipzig, Germany, for financial support.

\bibliographystyle{alpha}
\bibliography{bib}

\end{document}